\documentclass[12pt,a4paper,reqno]{amsart}
\usepackage{amsfonts,amsthm,amsmath,amscd}
\usepackage[utf8]{inputenc}
\usepackage{t1enc}
\usepackage[mathscr]{eucal}
\usepackage{indentfirst}
\usepackage{tikz}
\usepackage{caption}
\usepackage{subcaption}
\usepackage[font=small,labelfont=bf]{caption}
\usetikzlibrary{automata,arrows,positioning,calc}
\usepackage{graphicx,graphics,pict2e}
\usepackage{tkz-berge}
\usepackage{float}
\usepackage{epic}
\usepackage{lipsum}
\usepackage{stmaryrd}
\usepackage{authblk}
\usepackage[symbol]{footmisc}
\usepackage[normalem]{ulem}
\numberwithin{equation}{section}
\usepackage[margin=2.9cm]{geometry}
\usepackage{epstopdf}
\RequirePackage{doi}
\usepackage{hyperref}
\allowdisplaybreaks
\usepackage{cite}

\usepackage{verbatim,amssymb,amsfonts}
\usepackage{mathrsfs}
\usepackage{mathtools}

\usepackage{tikz-cd}
\usepackage{indentfirst}
\usepackage{slashed}
\usepackage{thmtools}
\usepackage{setspace}
\usepackage{enumerate}
\usepackage{accents}

\theoremstyle{plain}
\newtheorem{theorem}{Theorem}[section]

\newtheorem{lemma}[theorem]{Lemma}
\newtheorem{corollary}[theorem]{Corollary}
\newtheorem{proposition}[theorem]{Proposition}

\theoremstyle{definition}
\newtheorem{definition}[theorem]{Definition}
\newtheorem{remark}[theorem]{Remark}
\newtheorem{example}[theorem]{Example}

\DeclarePairedDelimiterX{\inp}[2]{\langle}{\rangle}{#1, #2}

\newcommand{\cA}{{\mathcal A}}

\newcommand{\cD}{{\mathcal D}}
\newcommand{\cE}{{\mathcal E}}

\newcommand{\cL}{{\mathcal L}}

\newcommand{\cV}{{\mathcal V}}

\newcommand{\ba}{\begin{eqnarray}}
	\newcommand{\na}{\end{eqnarray}}
\newcommand{\ban}{\begin{eqnarray*}}
	\newcommand{\nan}{\end{eqnarray*}}

\newcommand{\C}{{\mathbb C}}

\newcommand{\R}{{\mathbb R}}
\newcommand{\Z}{{\mathbb Z}}

\renewcommand{\thefootnote}{\fnsymbol{footnote}}

\makeatletter
\g@addto@macro{\endabstract}{\@setabstract}
\newcommand{\authorfootnotes}{\renewcommand\thefootnote{\@fnsymbol\c@footnote}}%
\makeatother

\makeatletter
\@namedef{subjclassname@2020}{%
	\textup{2020} Mathematics Subject Classification}
\makeatother

\title[]{On the Spectrum and Energy of  \\ Digraphs with Self-Loops}



\begin{document}
	
	\begin{center}
		\vspace{-1cm}
		\maketitle
		
		\normalsize
		\authorfootnotes

		Kevin Fung, Johnny Lim\footnote[1]{Corresponding author.}
		\par \bigskip

		\small{School of Mathematical Sciences, Universiti Sains Malaysia, Penang, Malaysia}\par \bigskip

	\end{center}
	
	\address{School of Mathematical Sciences, Universiti Sains Malaysia, Penang, Malaysia
	}
	\email{fungkevin@student.usm.my}
	\email{johnny.lim@usm.my}
	
	\begin{abstract}
		A digraph with self-loops $D_S$ with vertex set $\cV$ is a simple digraph with a self-loop attached at every vertex in $S\subseteq \cV.$ 
		We investigate the adjacency energy $E(D_S)$ and properties, which generalize several classical results on simple directed graphs. We establish a sufficient condition and a necessary condition to show that the presence of self-loops fundamentally changes its interaction with strong components: if $D_1,...,D_k$ are the strong components of $D_S$, then the strict subadditivity $$E(D_S) < \sum^k_{i=1} E(D_i)$$ can occur when $S \neq \emptyset,$ whereas for the loopless (undirected or directed) graphs equality always holds. We also establish bounds and extremal characterizations for $E(D_S)$ and the spectral radius $\rho(D_S),$ including the McClelland type and Gudi\~{n}o-Rada bounds. Finally, we introduce the complement of digraphs $\overline{D_S}$ that incorporates the loops, which allows the study of Nordhaus-Gaddum relation for $\rho(D_S).$ 
		For regular digraphs, we obtain a formula for $E(D_S)+E(\overline{D_S})$ in terms of the adjacency spectrum of $D_S$.
	\end{abstract} 
	
	\section{Introduction}
	\label{sec1}
	
	A digraph with self-loops $D_S$ is a pair $(\cV,\cA)$, where $\cV$ is a nonempty finite set of elements called \textit{vertices} and $\cA$ is a (possibly empty) set of ordered pairs of vertices called \textit{arcs}. Let $\cA'$ be the subset of $\cA$ that consists of ordered pairs of \textit{distinct} vertices. Let $S$ be a subset of $\cV$ that consists of vertices $v$ such that $(v,v)\in \cA$. We call $\cV,$ $\cA,$ and $S$ the \textit{vertex set, arc set, and self-loop set} of $D_S,$ respectively.  If $|\cV|=n, |\cA'|=m$ and $|S|=\sigma$, we shall say $D_S$ is a digraph of order $n$, size $m$ and has $\sigma$ self-loops. If $S= \emptyset$, then $D_S=D$ is called \textit{simple digraph}. If $S=V$, then we say $D_S=\widetilde{D}$ is a \textit{digraph with full loops}.
	
	For $(v,w)\in \cA,$ if $v \neq w,$ we call $(v,w)$ an arc from $v$ to $w.$ If $v = w,$ then such arc is a directed self-loop at $v$. Note that by our definition of $\cA$, multiple arcs and multiple loops are not allowed. Therefore, for two distinct vertices $v,w$, there can be at most one arc from $v$ to $w$ and at most one directed loop at each vertex $v$.
	
	Let $G=(\cV,\cE)$ be a finite simple undirected graph where $\cE$ is the set of undirected edges $vw$ for distinct $v,w \in \cV.$ By attaching a self-loop at every vertex in $S \subset \cV,$ we call such graph $G_S=(\cV,\cE, S)$ a self-loop (undirected) graph.  The \textit{symmetrization} of $G_S$, denoted as $\overleftrightarrow{G_S}$, is defined to be the directed graph with the same vertex set as $G_S$ and each edge $vw$ in $G_S$ is replaced with two arcs $(v,w)$ and $(w,v)$ for $v \neq w$ and each undirected self-loop $vv$ is replaced with a directed self-loop $(v,v)$.  
	
	For a digraph with self-loops $D_S=(\cV,\cA)$ with $\cV=\{v_1, \ldots, v_n\},$ the adjacency matrix $A(D_S)=(a_{ij})_{1 \leq i,j \leq n}$ is defined as
	\[
	a_{ij}=
	\begin{cases}
		0, &\text{if } (v_i,v_j) \notin \cA,\\
		1, &\text{if } (v_i,v_j) \in \cA.
	\end{cases}
	\]
	By our definition, the adjacency matrix $A(D_S)$ is a $(0,1)$-matrix but not necessary symmetric as in the case of undirected graphs. So the eigenvalues of $A(D_S)$ might be complex. Moreover, we have $A(G_S)=A(\overleftrightarrow{G_S})$.
	
	Let $D_S$ be a digraph with self-loops of order $n$ with vertex set $V=\{v_1,\ldots,v_n\}$ and adjacency matrix $A(D_S)=(a_{ij})_{1 \leq i,j \leq n}$. For each $i$, the out-degree (resp. in-degree) $deg^+(v_i)$ (resp. $deg^-(v_i)$) of vertex $v_i$ is defined as 
	\begin{align*}
		deg^+ (v_i) &= \sum_{j=1}^n a_{ij}, \\
		deg^- (v_i) &= \sum_{j=1}^n a_{ji}.
	\end{align*}
	Hence, one can interpret the out-degree of $v_i$ as the $i$-th row sum and in-degree as the $i$-th column sum of $A(D_S)$. Notice that each self-loop contributes $1$ to both the out- and in-degree of vertex, respectively. Meanwhile, for a graph with self-loops $G_S=(\cV, \cE, S)$ with $\cV=\{v_1,\ldots,v_n\}$ and adjacency matrix $A(G_S)$, the degree of $v_i$ $deg_{G_S}(v_i)$ is defined as 
	\[
	deg_{G_S}(v_i)=
	\begin{cases}
	\displaystyle	\sum_{j=1}^n a_{ij} +1, &\text{if $v_i\in S$,}\\
	\displaystyle	\sum_{j=1}^n a_{ij}, &\text{if $v_i \in \cV-S$.}
	\end{cases}
	\]
	Our definition here implies that each self-loop contributes $2$ to the degree of a vertex. More specifically, if $deg_G(v_i)$ denote the degree of vertex $v_i$ without considering loops, then $deg_{G_S}(v_i)=deg_{G}(v_i)+2$ for $v_i\in S$. A graph with self-loops $G_S$ is said to be \textit{$(a,b)$-bidegreed} if for every vertex $v$, $deg_{G_S}(v)$ equals to either $a$ or $b$.
	
	A digraph with self-loops $D_S$ is called \textit{$r$-regular} if 
	\begin{equation}
		\label{eq:regular1}
		deg^-(v)=deg^+(v)=r
	\end{equation}
	holds for every vertex $v$ of $D_S$.    
	If $D_S$ is $r$-regular, then the spectral radius of $D_S$ equals to $r$, with the corresponding eigenvector $j=(1,\ldots,1)^T$. If $\sigma=0$, then the definition reduces to the definition of regularity of simple digraphs. If $\sigma=n$, then the underlying simple digraph has regularity of $r-1$. 
	
	We adopt the following convention throughout: $[n]$ as the set $\{1,2,\ldots, n\};$  $I_n$ as the identity matrix of order $n;$ $J_n$ as the matrix of order $n$ whose entries are all 1's; $K_n$ as the complete graphs of order $n$; and $K_{a \times b}$ as the complete multipartite graph with $a$ parts, each of size $b$. A digraph with self-loops is said to be acyclic if it does not contain any cycle of length at least 2. 	
	
	A directed $2$-cycle consists of two distinct vertices $v_i$ and $v_j$ with both arcs $(v_i,v_j)$ and $(v_j,v_i)$. Throughout this paper, we count each directed $2$-cycle once for each choice of starting vertex. Thus the closed walks $v_i v_j v_i$ and $v_j v_i v_j$ are counted separately, and their total count over all directed $2$-cycles is denoted by $c_2$. Hence,
	\[
	c_2 = \sum_{i \neq j} a_{ij}a_{ji}
	= \mathrm{Tr}(A(D_S)^2) - \sigma.
	\]
	The phrase ``$c_2$ $2$-cycles'' always refers to this count.
	A digraph with self-loops $D_S$ is called \textit{strongly connected} if for any two distinct vertices $v,w$ of $D_S$, there is a path from $v$ to $w$ and vice versa. A \textit{strong component} of $D_S$ is a maximally strongly connected sub-digraph of $D_S$. For convention not mentioned here, we refer the readers to \cite{cvetkovic1995spectra}.
	
	Motivated by Gutman et al. \cite{gutman2021energy}, in Sect. 2, we define the energy $E(D_S)$ of a digraph with self-loops $D_S$ as 
	\begin{equation*}
		E(D_S)=\sum_{i=1}^n \left | \Re(\lambda_i)-\frac{\sigma}{n} \right|,
	\end{equation*}
	where $\Re(\lambda_i)$ is the real part of $\lambda_i \text{ for } 1 \leq i \leq n$. If $S=\emptyset$, then $E(D_S)$ reduces to the energy $E(D)$ of simple directed graph given by Pe\~{n}a and Rada \cite{pena2008energy}. In \cite{rada2009mcclelland}, Rada showed that if $D$ is a simple digraph with $D_1, \ldots, D_k$ being its strong components, then $E(D)=\sum_{i=1}^k E(D_i)$ holds. In Sect. \ref{sec3}, we extend this to the case of digraphs with self-loops and give a sufficient condition and a necessary condition for $E(D_S)\leq \sum_{i=1}^k E(D_i)$ to hold. In particular, an example is given to demonstrate that $E(D_S) < \sum_{i=1}^k E(D_i)$ is possible for $S \neq \emptyset$. In Sect. \ref{sec4}, for a digraph $D_S$ of order $n$, size $m,$ and $\sigma$ self-loops, the McClelland type bound 
	$$E(D_S)\leq \sqrt{\frac{1}{2}n \left(m+c_2+2\sigma-\frac{2\sigma^2}{n} \right)},$$ 
	as well as the characterization of the equality that generalizes that of Rada \cite{rada2009mcclelland} is established. 
	
In Sect. \ref{sec5}, a lower bound for the spectral radius of $D_S$, namely $\rho\geq \frac{c_2+\sigma}{n}$, and its characterization are given. This lower bound generalizes \cite[Theorem 2.1]{gudino2010lowerboundofspectral}. 	
	We introduce the complement $\overline{D_S}$ of $D_S$ and establish the relationship between their spectrum in the case of regular digraphs. 
	(See Lemma \ref{lambda_complement_regular}). For a digraph $D_S$ of order $n$, size $m$ and $\sigma$ self-loops and $\overline{D_S}$ with size $\overline{m}$, if $\rho$ and $\overline{\rho}$ are their spectral radius, respectively, then in Lemma \ref{ul_bounds_p+p-}, we establish the Nordhaus-Gaddum bound for spectral radius:
	\begin{equation*}
		1-\delta_\sigma+\frac{c_2+\overline{c_2}}{n} \leq \rho+ \overline{\rho} \leq 1-\delta_\sigma + \sqrt{ (n-1)^2-\frac{4\sigma(n-\sigma)}{n^2} + \frac{4\sigma(n-\sigma)}{n}+\frac{(n-1)(c_2+\overline{c_2})}{n} }.
	\end{equation*}
	Finally, a formula for $E(D_S)+E(\overline{D_S})$ is given for regular $D_S$ (See Theorem \ref{sum_of_E_and_E_comp}), followed by a fundamental yet important example to illustrate the theorem.
	\section{Preliminaries }
	\label{sec2}
	The following results will be used frequently in this paper.
	
	\begin{theorem}[Perron-Frobenius Theorem] \cite{perron1907, frobenius1912} \cite[Theorems~8.2.8, 8.3.1 and 8.4.4]{HornJohnson2013}
		\label{Perron}
		Let $A$ be a nonnegative real matrix. Then there exists a   nonnegative real number $\rho$ such that the following holds:
		\begin{enumerate}[(a)]
			\item If $\lambda$ is any (possibly complex) eigenvalue  of $A$, then $|\lambda|\leq \rho$ holds. The inequality is strict  for $\lambda\neq\rho$ if all entries of $A$ are positive (in this case we say $A$ is positive).
			
			\item If $A$ is positive, then $\rho$ is a simple eigenvalue (algebraic multiplicity of $\rho$ equals one).
			
			\item There is a nonzero eigenvector corresponding to $\rho$ with nonnegative entries. If $A$ is positive, then there is an eigenvector corresponding to $\rho$ with positive entries.
		\end{enumerate}
		If $A$ is irreducible, then $\rho$ is a simple eigenvalue and has positive left and right eigenvectors.
	\end{theorem}
	
	\begin{lemma}
		\cite{1996herstein_abstract_algebra}
		\label{rationalrootsofmonic}
		Let $p(x)\in \Z[x]$ be a monic polynomial. If $a$ is a rational root of $p(x)$, then $a$ is an integer.
	\end{lemma}
	
	\begin{lemma}
		\label{charpoly_arc_deletion}
		(Compare with \cite[Theorem 2.3]{rada2009mcclelland})
		Suppose $D_S$ is a digraph with self-loops of order $n$. Let $D_S'$ be a sub-digraph of $D_S$ obtained by deleting some arcs that do not belong to any cycle. Then, $D_S$ and $D_S'$ have the same characteristic polynomial, i.e.
		$\phi_{D_S}(\lambda)=\phi_{D_S'}(\lambda)$.
	\end{lemma}
	
	\begin{proof}
		Suppose $\phi_{D_S}(\lambda)=\lambda^n+a_1\lambda^{n-1}+\cdots+a_n$. By the coefficients theorem for digraphs\cite[Theorem 1.2]{cvetkovic1995spectra}, we have 
		\begin{align*}
			a_i=\sum_{L\in \cL_i} (-1)^{p(L)},
		\end{align*}
		where $\cL_i$ is the set of all linear sub-digraphs of $D_S$ with exactly $i$ vertices and $p(L)$ is the number of strong components of $L$. 
		
		The coefficients of the characteristic polynomial depend only on the cycle $C_k$ for $k \geq 1$ (with loops being regarded as $C_1$). Hence, deleting arcs that do not belong to any cycle would not affect the coefficients of $\phi_{D_S}(\lambda)$. Therefore, $\phi_{D_S}(\lambda)=\phi_{D_S'}(\lambda)$.
		
	\end{proof}
	
	Now we extend the definition of energy of simple digraphs by Pe\~{n}a and Rada \cite{pena2008energy} to digraphs with self-loops.
	\begin{definition}
		\label{E(D_S)}
		Let $D_S$ be a digraph of order $n$, size $m$ and $\sigma$ self-loops, $0 \leq \sigma \leq n$.
		The energy of $D_S$, denoted as $E(D_S),$ is defined as 
		\[
		E(D_S)=\sum_{i=1}^n \left | \Re(\lambda_i)-\frac{\sigma}{n} \right|,
		\]
		where $\Re(\lambda_i)$ is the real part of $\lambda_i \text{ for } 1 \leq i \leq n$.
	\end{definition}
	
	\section{Some Properties of Energy of Digraphs with Self-Loops}
	\label{sec3}
	
	It is well known (cf. \cite[Theorem 1.9]{cvetkovic1995spectra}) that for a digraph with order $n$, we have
	$\sum_{i=1}^n \lambda^k_i = \sum_{i=1}^n (A(D)^k)_{ii}= w_k$, where $w_k$ is the number of closed walks of length $k$. In the case of digraphs with self-loops, we have the following results.
	
	\begin{lemma}
		\label{sum_eigenvalue}
		Let $D_S$ be a digraph of order $n$ with $\sigma$ self-loops and $c_2$ 2-cycles. Then,
		\begin{enumerate}[(a)]
			\item $\displaystyle \sum_{i=1}^n \lambda_i(D_S) = \sigma.$
			\item $\displaystyle \sum_{i=1}^n \lambda^2_i(D_S) = c_2+\sigma.$
		\end{enumerate}
	\end{lemma}
	
	\begin{proof}
		Part (a) is clear as $\sum_{i=1}^n \lambda_i =  \sum_{i=1}^n (A(D_S))_{ii} = \sigma.$ Part (b) follows immediately by noting that closed walks of length two are either in the form of $vvv$ (self-loop) or $vwv$ (2-cycle) for distinct vertices $v,w$. 
	\end{proof}
	
	From Definition \ref{E(D_S)}, we have the following straightforward lemma:
	\begin{lemma}\label{lemma_E(D_S)=E(D)}
		Suppose $D_S$ is a digraph with order $n$ and $\sigma$ self-loops. Then 
		$E(D_S)=E(D)$ if $\sigma=0$ or $\sigma=n$, where $E(D)$ is the energy of simple digraph defined in \cite{pena2008energy}.
	\end{lemma}
	
	\begin{proof}
		The case $\sigma=0$ is trivial. Suppose $\sigma=n$. Then $A(D_S)=A(D)+I_n$. So $\lambda(A(D_S))=\lambda(A(D))+1$. Thus,
		\begin{align*}
			E(D_S) &= \sum_{i=1}^n \left | \Re(\lambda_i(A(D_S)))-\frac{n}{n} \right|\\           
			&= \sum_{i=1}^n \left | \Re(\lambda_i(A(D)))+1-1 \right| \\        
			&= \sum_{i=1}^n \left | \Re(\lambda_i(A(D))) \right| \\            
			&= E(D). \qedhere
		\end{align*}
		
	\end{proof}
	
	\begin{proposition}\label{energy+}
		Let $D_S$ be a digraph of order $n$ and $\sigma$ self-loops. Then, the energy of $D_S$ can be expressed as
		\begin{equation}\label{eq_energy_2_times}
			E(D_S) = 2\sum_{\Re(\lambda_i)> \frac{\sigma}{n}} \left( \Re(\lambda_i)-\frac{\sigma}{n} \right).
		\end{equation}

	\end{proposition}
	
	\begin{proof}
		By writing $\lambda_i= \Re(\lambda_i) + i \Im(\lambda_i),$  it follows from Lemma \ref{sum_eigenvalue}(a) that 
		\[
		\sum^n_{i=1} \left(\Re(\lambda_i)+i\Im(\lambda_i) \right)= \sigma.
		\]
		Since every complex eigenvalue comes with a conjugate pair eigenvalue, it further implies that $\sum^n_{i=1} \Re(\lambda_i)=\sigma.$ Therefore, 
		\[
		\sum_{i=1}^n \left( \Re(\lambda_i) - \frac{\sigma}{n} \right)=0.
		\]
		Since 
		\[
		\sum_{i=1}^n \left( \Re(\lambda_i) - \frac{\sigma}{n} \right) 
		=	\sum_{\Re(\lambda_i)> \frac{\sigma}{n}} \left( \Re(\lambda_i)-\frac{\sigma}{n} \right) + \sum_{\Re(\lambda_i) \leq \frac{\sigma}{n}} \left( \Re(\lambda_i)-\frac{\sigma}{n} \right),
		\]
		it holds that
		\[
		\sum_{\Re(\lambda_i)> \frac{\sigma}{n}} \left( \Re(\lambda_i)-\frac{\sigma}{n} \right) = - \sum_{\Re(\lambda_i) \leq \frac{\sigma}{n}} \left( \Re(\lambda_i)-\frac{\sigma}{n} \right).
		\]
		Thus, 
		\[
		E(D_S) = 2\sum_{\Re(\lambda_i)> \frac{\sigma}{n}} \left( \Re(\lambda_i)-\frac{\sigma}{n} \right).
		\]
	\end{proof}

	\begin{remark}
		Since $\sum_{i=1}^n \left ( \Re(\lambda_i)+i\Im(\lambda_i) \right )^2 = \sum_{i=1}^n \left( \Re(\lambda_i)^2-\Im(\lambda_i)^2 + 2i\Re(\lambda_i)\Im(\lambda_i)  \right),$ it follows from Lemma~\ref{sum_eigenvalue}(b) that
		\begin{equation}\label{eq:ReImineq1}
			\sum^n_{i=1} \left( \Re(\lambda_i)^2-\Im(\lambda_i)^2 \right) = c_2+\sigma. 
		\end{equation}
	\end{remark}		
	
	Moreover, we have the following inequality:
	\begin{proposition}
		
		\label{sum_of_re&im}
		Let $D_S$ be a digraph of order $n,$ size $m$ and $\sigma$ self-loops. Then,
		\begin{equation}\label{eq:ReImineq2}
			\sum_{i=1}^n  \left(\Re(\lambda_i)^2+\Im(\lambda_i)^2 \right)  \leq m+\sigma.
		\end{equation}
	\end{proposition} 
	
	\begin{proof}
		By the Schur triangularization theorem (cf. \cite[Theorem 2.3.1(a)]{HornJohnson2013}), the adjacency matrix $A(D_S)$ can be decomposed as $A(D_S)=U^*TU$ for some unitary matrix $U$ and upper triangular matrix $T$ with diagonal entries being the eigenvalues of $A(D_S)$.\\
		Thus, 
		\begin{align*}
			m+\sigma &=      \sum_{1\leq i,j \leq n} a_{ij}     
			=      \sum_{1\leq i,j \leq n} |a_{ij}|^2   
			=      \sum_{1\leq i,j \leq n} |t_{ij}|^2 \\
			&\geq   \sum_{i=1}^n |t_{ii}|^2  
			=      \sum_{i=1}^n |\lambda_{i}|^2
			= \sum_{i=1}^n \left ( \Re(\lambda_i)^2+\Im(\lambda_i)^2 \right). \qedhere  
		\end{align*}
	\end{proof}
	
	It follows immediately from equations \eqref{eq:ReImineq1} and \eqref{eq:ReImineq2} that:
	
	\begin{corollary} \label{ReImineq}
		Let $D_S$ be a digraph of order $n,$ size $m$, with $\sigma$ self-loops and $c_2$ 2-cycles, and with eigenvalues $\lambda_i, 1 \leq i \leq n.$  Then,
		\begin{align*}
			\sum_{i=1}^n  \Re(\lambda_i)^2 &\leq \frac{m + c_2 + 2\sigma}{2}, \\ 
			\sum_{i=1}^n  \Im(\lambda_i)^2 &\leq \frac{m - c_2}{2}.
		\end{align*}
		
	\end{corollary}
	
	For a digraph with self-loops $D_S$, let $D_1, D_2,..., D_k$ be its strong components, each with order $n_i$ and $\sigma_i$ self-loops. Then by \cite[Theorem 2.4]{cvetkovic1995spectra},
	\[
	\phi_{D_S}(\lambda)=\phi_{D_1}(\lambda)\cdot \phi_{D_2}(\lambda) \cdot \cdots \cdot \phi_{D_k}(\lambda)
	\]
	holds. Thus, the spectrum of $D_S$ is the union of the spectrum of its strong components. For each $i=1,\ldots, k$, let $\{\lambda_{i,j} \}_{j=1}^{n_i}$ be the eigenvalues of $A(D_i)$, then the energy of $D_S$ can be written as 
	\begin{equation}\label{energy_eq}
		E(D_S)= \sum_{i=1}^k \sum_{j=1}^{n_i} \left| \Re(\lambda_{i,j})-\frac{\sigma}{n} \right|.    
	\end{equation}
	
	While the energy of each $D_i$ is 
	\begin{equation}\label{sum_energy_components_eq}
		E(D_i)= \sum_{j=1}^{n_i} \left| \Re(\lambda_{i,j})-\frac{\sigma_i}{n_i} \right|.
	\end{equation}
	For $S=\emptyset$, Rada \cite{rada2009mcclelland} showed that 
	$$E(D)=E(D_1)+ E(D_2)+ \cdots + E(D_k).$$
	However, it is not generally true that 
	$$E(D_S)=E(D_1)+ E(D_2)+ \cdots + E(D_k)$$
	holds for nonempty $S$.
	
	We give a simple example in which $E(D_S)<\sum_{i=1}^{k} E(D_i)$ is possible.
	\begin{example}
		Consider $D_S=\overleftrightarrow{K_2^+} \cup (\overrightarrow{C_3})_{S'}$, where $\overleftrightarrow{K_2^+}$ is symmetrization of the complete graph $K_2$ with one self-loop and $(\overrightarrow{C_3})_{S'}$ is a cycle digraph of length three with two self-loops. It can be shown that 
		\begin{align*}
			\mathrm{Spec}(\overleftrightarrow{K_2^+}) &= \left\{\frac{1+\sqrt{5}}{2}, \frac{1-\sqrt{5}}{2} \right\}, \\
			\mathrm{Spec}\left((\overrightarrow{C_3})_{S'}\right) &= \left\{ 1.7549, 0.1226+0.7449i, 0.1226-0.7449i \right\}.
		\end{align*}
		By \cite[Theorem 2.4]{cvetkovic1995spectra}, 
		$$\mathrm{Spec}(D_S) =  \mathrm{Spec}\left(\overleftrightarrow{K_2^+} \right) \bigcup \mathrm{Spec}\left((\overrightarrow{C_3})_{S'}\right).$$
		By simple computation, we have 
		\begin{align*}
			E\left(\overleftrightarrow{K_2^+}\right) &=\sqrt{5},\\
			E\left((\overrightarrow{C_3})_{S'}\right) &\approx 2.1764,\\
			E(D_S) &\approx 4.3458 < 4.4125 \approx E\left(\overleftrightarrow{K_2^+}\right)+E\left((\overrightarrow{C_3})_{S'}\right). 
		\end{align*}
	\end{example}
	
	
	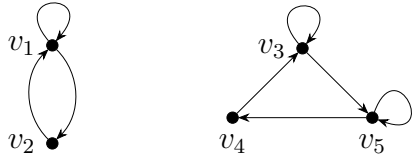
\begin{figure}[h]
		\vspace{-1em}
		\begin{minipage}[b]{0.23\linewidth}
			\centering
			\begin{tikzpicture}[node distance={13mm}, very thin, main/.style = {draw, circle, fill, inner sep=1.5pt, minimum size=0pt}] 
				\node[main] (1) [label=left:$v_1$]{}; 
				\node[main] (2) [below of=1]
				[label=left:$v_2$]{};
				
				\draw[-Stealth] (1) to [bend left =45](2); 
				\draw[-Stealth] (2) to [bend left =45] (1); 
				\draw[-Stealth] (1) to [out=130,in=50,looseness=22] (1);
			\end{tikzpicture}
		\end{minipage}
		\begin{minipage}[b]{0.23\linewidth}
			\centering
			\begin{tikzpicture}[node distance={13mm}, very thin, main/.style = {draw, circle, fill, inner sep=1.5pt, minimum size=0pt}] 
				\node[main] (1) [label=left:$v_3$]{}; 
				\node[main] (2) [below left of=1]
				[label=below :$v_4$]{};
				\node[main] (3) [below right of=1][label=below:$v_5$]{};
				
				\draw[-Stealth] (1) -- (3);
				\draw[-Stealth] (3) -- (2); 
				\draw[-Stealth] (2) -- (1);             
				\draw[-Stealth] (1) to [out=130,in=50,looseness=22] (1);
				\draw[-Stealth] (3) to [out=70,in=-20,looseness=20] (3);  
			\end{tikzpicture}
		\end{minipage}
		\caption{$D_S=\overleftrightarrow{K_2^+} \cup (\overrightarrow{C_3})_{S'}$.}
	\end{figure}
	
	Nevertheless, we have the following result.
	\begin{proposition}
		Let $D_S$ be a digraph of order $n$ and $\sigma$ self-loops. Let $D_i, 1 \leq i \leq k,$ be its strong components, each with order $n_i$ and $\sigma_i$ self-loops. Then
		$$\left|E(D_S)-\sum_{i=1}^k E(D_i) \right| \leq 2\sigma.$$
	\end{proposition}
	
	\begin{proof}
		For fixed $i$, using the triangle inequality,
		\begin{align}\label{streq:1}
			\sum_{j=1}^{n_i} \left| \Re(\lambda_{i,j})\right|
			&= \sum_{j=1}^{n_i} \left| \Re(\lambda_{i,j})-\frac{\sigma}{n} + \frac{\sigma}{n}  \right| \nonumber\\
			&\leq \sum_{j=1}^{n_i} \left( \left| \Re(\lambda_{i,j})-\frac{\sigma}{n}  \right| + \left| \frac{\sigma}{n}  \right| \right) \nonumber \\
			&= \sum_{j=1}^{n_i} \left| \Re(\lambda_{i,j})-\frac{\sigma}{n}  \right| + \frac{\sigma}{n}n_i.
		\end{align}
		
		On the other hand, using the reverse triangle inequality,
		\begin{align}\label{streq:2}
			\sum_{j=1}^{n_i} \left| \Re(\lambda_{i,j})\right| &= \sum_{j=1}^{n_i} \left| \Re(\lambda_{i,j})-\frac{\sigma_i}{n_i} + \frac{\sigma_i}{n_i}  \right|\nonumber \\
			&\geq \sum_{j=1}^{n_i} \left( \left| \Re(\lambda_{i,j})-\frac{\sigma_i}{n_i}  \right| - \left| \frac{\sigma_i}{n_i} \right| \right) \nonumber\\
			&=E(D_i) - \sigma_i.
		\end{align}
		Summing \eqref{streq:1} and \eqref{streq:2} over all $i,$  we obtain
		\begin{align*}
			E(D_S)+\sigma 
			&\geq \sum_{i=1}^k \sum_{j=1}^{n_i} \left| \Re(\lambda_{i,j})\right|
			\geq \sum_{i=1}^k \left[ E(D_i) - \sigma_i \right] 
			= \sum_{i=1}^k E(D_i) - \sigma.
		\end{align*}
		Hence, 
		\[
		E(D_S) -\sum_{i=1}^k E(D_i) \geq -2\sigma.
		\]
		By replacing $\dfrac{\sigma}{n}$ with $\dfrac{\sigma_i}{n_i}$ (and vice versa) in each inequality, we have 
		\[
		2\sigma \geq E(D_S) -\sum_{i=1}^k E(D_i).
		\]
		Therefore, 
		\[
		\left|E(D_S)-\sum_{i=1}^k E(D_i) \right| \leq 2\sigma. \qedhere
		\]  
	\end{proof}
	
	The following theorem features a sufficient condition on when the energy of a digraph with self-loops is less than or equal to the sum of energies of its strong components:
	
	\begin{theorem}\label{suffcond1}
		Suppose $D_S$ is a digraph of order $n$ with $\sigma$ self-loops. Let $D_i, 1 \leq i \leq k,$ be its strong components, each with order $n_i$ and $\sigma_i$ self-loops, such that $n=\sum_{i=1}^k n_i$ and $\sigma=\sum_{i=1}^k \sigma_i$. Let $1\leq l <k$. Suppose $\frac{\sigma_i}{n_i} > \frac{\sigma}{n}$ holds for $i=1, \ldots, l$ and $\frac{\sigma_i}{n_i} \leq \frac{\sigma}{n}$ holds for $i=l+1, \ldots, k$. Let 
		\begin{align*}
			A_i &=\left\{ j\in[n_i] :      \Re(\lambda_{i,j})> \frac{\sigma}{n}\right\}, \\
			B_i &=\left\{ j\in[n_i] : \Re(\lambda_{i,j})> \frac{\sigma_i}{n_i} \right\}.
		\end{align*}  
		Then the following implication
		\begin{equation*}
			\sum_{i=1}^l \left( \frac{\sigma_i}{n_i}-\frac{\sigma}{n}\right) |A_i| \leq \sum_{i=l+1}^k \left( \frac{\sigma}{n}-\frac{\sigma_i}{n_i} \right)|A_i| \implies E(D_S) \leq \sum_{i=1}^k E(D_i)
		\end{equation*}
		holds. Moreover, if the inequality on the left hand side of the implication is strict, then the inequality on the right hand side will also be strict.
	\end{theorem}
	
	Before we start the proof, note that it is impossible for $\frac{\sigma_i}{n_i} < \frac{\sigma}{n}$ to hold for all $i\in [k]$ since this would implies $\sigma=\sum_{i=1}^k \sigma_i < \frac{\sigma}{n}  \sum_{i=1}^k n_i= \sigma$.
	Similarly for $\frac{\sigma_i}{n_i} > \frac{\sigma}{n}$. However, it is possible that $\frac{\sigma_i}{n_i} = \frac{\sigma}{n}$ holds for all $i \in [k]$ (For example, $\frac{n}{2} \widetilde{K_2}$). This justifies the restriction on $l$.
	
	\begin{proof}
		From the definitions of $A_i$ and $B_i,$ we have $B_i \subset A_i$ for $i=1,\ldots,l$ and $A_i \subset B_i$ for $i=l+1, \ldots, k$.
		From equation \eqref{eq_energy_2_times},
		\begin{align*}
			E(D_S)
			&= \sum_{i=1}^k \left[ 2\sum_{j\in A_i} \left( \Re(\lambda_{i,j})-\frac{\sigma}{n} \right) \right]\\
			&= 2\left[\sum_{i=1}^l  \sum_{j\in A_i} \left( \Re(\lambda_{i,j})-\frac{\sigma}{n} \right)  + \sum_{i=l+1}^k  \sum_{j\in A_i} \left( \Re(\lambda_{i,j})-\frac{\sigma}{n} \right) \right]\\
			&= 2 \sum_{i=1}^l  \left[  \left(\sum_{j\in B_i} + \sum_{j\in A_i-B_i}  \right) \left( \Re(\lambda_{i,j}) -\frac{\sigma_i}{n_i} +\frac{\sigma_i}{n_i}-\frac{\sigma}{n} \right) \right] \\
			&\quad + 2\sum_{i=l+1}^k \left[ \sum_{j\in A_i} \left( \Re(\lambda_{i,j})-\frac{\sigma}{n} \right) 
			+\sum_{j\in B_i-A_i} \left( \Re(\lambda_{i,j})-\frac{\sigma}{n} \right)-\sum_{j\in B_i-A_i} \left( \Re(\lambda_{i,j})-\frac{\sigma}{n} \right)\right] \\
			&= 2\sum_{i=1}^l \left[  \sum_{j\in B_i} \left( \Re(\lambda_{i,j}) -\frac{\sigma_i}{n_i}\right) + \sum_{j\in A_i-B_i} \left( \Re(\lambda_{i,j}) -\frac{\sigma_i}{n_i} \right)  + \sum_{j\in A_i} \left( \frac{\sigma_i}{n_i}-\frac{\sigma}{n} \right)   \right]\\
			&\quad+ 2\sum_{i=l+1}^k \left[ \sum_{j\in B_i} \left( \Re(\lambda_{i,j})-\frac{\sigma_i}{n_i} \right) + \sum_{j\in B_i} \left( \frac{\sigma_i}{n_i}-\frac{\sigma}{n} \right)-\sum_{j\in B_i-A_i} \left( \Re(\lambda_{i,j})-\frac{\sigma}{n} \right)  \right],
		\end{align*}
		
		where the fourth equality follows from
		\begin{align*}
			\sum_{j\in A_i} &\left( \Re(\lambda_{i,j})-\frac{\sigma}{n} \right) 
			+\sum_{j\in B_i-A_i} \left( \Re(\lambda_{i,j})-\frac{\sigma}{n} \right) \\
			&= \sum_{j\in B_i} \left( \Re(\lambda_{i,j})-\frac{\sigma}{n} \right) 
			=\sum_{j\in B_i} \left( \Re(\lambda_{i,j})-\frac{\sigma_i}{n_i} \right) + \sum_{j\in B_i} \left( \frac{\sigma_i}{n_i}-\frac{\sigma}{n} \right).
		\end{align*}
		Since $\sum_{i=1}^k E(D_i)= 2 \left[\sum_{i=1}^l  \sum_{j\in B_i} \left( \Re(\lambda_{i,j}) -\frac{\sigma_i}{n_i}\right) +  \sum_{i=l+1}^k  \sum_{j\in B_i} \left( \Re(\lambda_{i,j})-\frac{\sigma_i}{n_i} \right) \right],$ we have
		\begin{align*}
			E(D_S) 
			&= \sum_{i=1}^k E(D_i) + 2\sum_{i=1}^l  \sum_{j\in A_i-B_i} \left( \Re(\lambda_{i,j}) -\frac{\sigma_i}{n_i} \right) + 2\sum_{i=1}^l \left( \frac{\sigma_i}{n_i}-\frac{\sigma}{n} \right) |A_i| \\
			&\hspace{-2em} + 2\sum_{i=l+1}^k \left( \frac{\sigma_i}{n_i}-\frac{\sigma}{n} \right) |B_i| - 2\sum_{i=l+1}^k \sum_{j\in B_i-A_i} \left( \Re(\lambda_{i,j}) -\frac{\sigma_i}{n_i}\right) - 2\sum_{i=l+1}^k \sum_{j\in B_i-A_i} \left( \frac{\sigma_i}{n_i}-\frac{\sigma}{n}\right).
		\end{align*}	
		For $i=1,\ldots,l$ and $j\in A_i-B_i, \frac{\sigma}{n}< \Re(\lambda_{i,j}) \leq \frac{\sigma_i}{n_i}$. So $2\sum_{i=1}^l  \sum_{j\in A_i-B_i} \left( \Re(\lambda_{i,j}) -\frac{\sigma_i}{n_i} \right) \leq 0$. Similarly, $-2\sum_{i=l+1}^k \sum_{j\in B_i-A_i} \left( \Re(\lambda_{i,j}) -\frac{\sigma_i}{n_i}\right) \leq 0$. Thus,   
		\begin{align*}
			E(D_S)
			&\leq \sum_{i=1}^k E(D_i)+0 + 2\sum_{i=1}^l \left( \frac{\sigma_i}{n_i}-\frac{\sigma}{n} \right) |A_i|\\
			&\quad + 2\sum_{i=l+1}^k \left( \frac{\sigma_i}{n_i}-\frac{\sigma}{n} \right) |B_i|+0 - 2\sum_{i=l+1}^k \left( \frac{\sigma_i}{n_i}-\frac{\sigma}{n} \right) [|B_i|-|A_i|]\\
			&=\sum_{i=1}^k E(D_i) + 2\sum_{i=1}^l \left( \frac{\sigma_i}{n_i}-\frac{\sigma}{n} \right)|A_i|+ 2\sum_{i=l+1}^k \left( \frac{\sigma_i}{n_i}-\frac{\sigma}{n} \right) |A_i|.
		\end{align*}
		By our assumption, $\sum_{i=1}^l \left( \frac{\sigma_i}{n_i}-\frac{\sigma}{n} \right)|A_i|+ \sum_{i=l+1}^k \left( \frac{\sigma_i}{n_i}-\frac{\sigma}{n} \right) |A_i|\leq 0$. Therefore, $E(D_S) \leq \sum_{i=1}^k E(D_i)$.
	\end{proof}
	
	The next theorem provides a necessary condition for $E(D_S)\leq \sum_{i=1}^k E(D_i)$ to hold.
	
	\begin{theorem}
		With the same assumption as in Theorem \ref{suffcond1}. The following implication  
		\begin{equation*}
			E(D_S)\leq \sum_{i=1}^k E(D_i) \implies \sum_{i=1}^l \left( \frac{\sigma_i}{n_i}-\frac{\sigma}{n} \right)|B_i| \leq \sum_{i=l+1}^k \left( \frac{\sigma}{n}-\frac{\sigma_i}{n_i} \right) |B_i|
		\end{equation*}
		holds. If the inequality on the left-hand side of the implication is strict, then the inequality on the right-hand side will also be strict.
	\end{theorem}
	
	\begin{proof}
		For $i=1,\ldots,l$, we have $\frac{\sigma_i}{n_i}-\frac{\sigma}{n}> 0$  and $\frac{\sigma_i}{n_i}-\frac{\sigma}{n} \leq 0$ for $i=l+1,\ldots,k$. Note that	
		\begin{align*}
			&E(D_S)-\sum_{i=1}^k E(D_i)\\
			&= 2   \sum_{i=1}^l \sum_{j\in A_i} \left( \Re(\lambda_{i,j})-\frac{\sigma}{n} \right)  + 2\sum_{i=l+1}^k  \sum_{j\in A_i} \left( \Re(\lambda_{i,j})-\frac{\sigma}{n} \right) \\
			&\qquad - 2\sum_{i=1}^l \sum_{j\in B_i} \left( \Re(\lambda_{i,j})-\frac{\sigma_i}{n_i} \right)  - 2\sum_{i=l+1}^k \sum_{j\in B_i} \left( \Re(\lambda_{i,j})-\frac{\sigma_i}{n_i} \right)\\
			&= 2\sum_{i=1}^l \sum_{j\in A_i-B_i} \left( \Re(\lambda_{i,j})-\frac{\sigma}{n} \right) + 2\sum_{i=1}^l \sum_{j\in B_i} \left( \Re(\lambda_{i,j})-\frac{\sigma}{n} \right)  + 2\sum_{i=l+1}^k \sum_{j\in A_i} \left( \Re(\lambda_{i,j})-\frac{\sigma}{n} \right) \\
			& -2\sum_{i=1}^l \sum_{j\in B_i} \left( \Re(\lambda_{i,j})-\frac{\sigma_i}{n_i} \right)  - 2\sum_{i=l+1}^k \sum_{j\in B_i-A_i} \left( \Re(\lambda_{i,j})-\frac{\sigma_i}{n_i} \right)  - 2\sum_{i=l+1}^k \sum_{j\in A_i} \left( \Re(\lambda_{i,j})-\frac{\sigma_i}{n_i} \right)\\
			&= 2\sum_{i=1}^l \sum_{j\in A_i-B_i} \left( \Re(\lambda_{i,j})-\frac{\sigma}{n} \right)  + 2\sum_{i=1}^l \sum_{j\in B_i}\left(\frac{\sigma_i}{n_i}-\frac{\sigma}{n} \right)
			- 2\sum_{i=l+1}^k\sum_{j\in B_i-A_i} \left( \Re(\lambda_{i,j}) -\frac{\sigma_i}{n_i} \right) \\ 
			&\qquad+  2\sum_{i=l+1}^k \sum_{j\in A_i} \left(\frac{\sigma_i}{n_i}-\frac{\sigma}{n} \right)\\
			&= 2\sum_{i=1}^l \sum_{j\in A_i-B_i} \left( \Re(\lambda_{i,j})-\frac{\sigma}{n} \right)  + 2\sum_{i=1}^l \sum_{j\in B_i}\left(\frac{\sigma_i}{n_i}-\frac{\sigma}{n} \right) 
			- 2\sum_{i=l+1}^k \sum_{j\in B_i-A_i} \left( \Re(\lambda_{i,j})-\frac{\sigma}{n} \right) \\
			&\qquad + 2\sum_{i=l+1}^k \sum_{j\in B_i-A_i}\left(\frac{\sigma_i}{n_i}-\frac{\sigma}{n} \right)  + 2\sum_{i=l+1}^k \sum_{j\in A_i}\left(\frac{\sigma_i}{n_i}-\frac{\sigma}{n} \right).
		\end{align*}
		Since 
		\[
		E(D_S)-\sum_{i=1}^k E(D_i) \leq 0
		\]
		and 
		\[
		2\sum_{i=1}^l \sum_{j\in A_i-B_i} \left( \Re(\lambda_{i,j})-\frac{\sigma}{n} \right)  - 2\sum_{i=l+1}^k \sum_{j\in B_i-A_i} \left( \Re(\lambda_{i,j})-\frac{\sigma}{n} \right)  \geq 0,
		\] we have 
		\[
		2\sum_{i=1}^l \sum_{j\in B_i}\left(\frac{\sigma_i}{n_i}-\frac{\sigma}{n} \right)  + 2\sum_{i=l+1}^k \sum_{j\in B_i-A_i}\left(\frac{\sigma_i}{n_i}-\frac{\sigma}{n} \right)  + 2\sum_{i=l+1}^k \sum_{j\in A_i}\left(\frac{\sigma_i}{n_i}-\frac{\sigma}{n} \right) \leq 0.
		\] 
		Therefore $\sum_{i=1}^l \left( \frac{\sigma_i}{n_i}-\frac{\sigma}{n} \right)|B_i| \leq \sum_{i=l+1}^k \left( \frac{\sigma}{n}-\frac{\sigma_i}{n_i} \right) |B_i|$.
	\end{proof}
	
	\begin{remark}
		Observe that 
		\begin{align*}
			&\sum_{i=1}^l \left( \frac{\sigma_i}{n_i}-\frac{\sigma}{n}\right) |A_i| - \sum_{i=l+1}^k \left( \frac{\sigma}{n}-\frac{\sigma_i}{n_i} \right)|A_i|\\
			&=\sum_{i=1}^l \left( \frac{\sigma_i}{n_i}-\frac{\sigma}{n} \right)|B_i| - \sum_{i=l+1}^k \left( \frac{\sigma}{n}-\frac{\sigma_i}{n_i} \right) |B_i| + \sum_{i=1}^l \left( \frac{\sigma_i}{n_i}-\frac{\sigma}{n} \right)\left(|A_i|-|B_i| \right) \\
			& \quad+ \sum_{i=l+1}^k \left( \frac{\sigma}{n}-\frac{\sigma_i}{n_i} \right) \left(|B_i|-|A_i|\right).
		\end{align*}
		Since the last two terms are greater or equals to zero, we have 
		\[
		\sum_{i=1}^l \left( \frac{\sigma_i}{n_i}-\frac{\sigma}{n}\right) |A_i| - \sum_{i=l+1}^k \left( \frac{\sigma}{n}-\frac{\sigma_i}{n_i} \right)|A_i|
		\geq \sum_{i=1}^l \left( \frac{\sigma_i}{n_i}-\frac{\sigma}{n} \right)|B_i| - \sum_{i=l+1}^k \left( \frac{\sigma}{n}-\frac{\sigma_i}{n_i} \right) |B_i|.
		\]
	\end{remark}
	
	The following lemma generalizes \cite[Proposition 2.1]{rada2009mcclelland}.
	\begin{lemma}
		\label{E(D_S)=0_acyclic}
		Let $D_S$ be a digraph with $\sigma$ self-loops.
		Then $E(D_S)=0$ if and only if $D_S$ is an acyclic digraph with either $\sigma=0$ or $\sigma=n$.
	\end{lemma}
	
	\begin{proof}
		The sufficient condition can be shown by direct calculation. We prove the necessary condition.
		Suppose $E(D_S)=0$. Then $\Re(\lambda_i)=\frac{\sigma}{n}$ for all $i=1,\ldots,n$. From $|\Re(\lambda_i)| \leq |\lambda_i| \leq \rho$, since $\rho$ is itself a real eigenvalue, this implies $\lambda_i=\frac{\sigma}{n}$ for all $i=1,\ldots,n$. By Lemma \ref{rationalrootsofmonic}, $\frac{\sigma}{n}$ must be an integer. So, $\sigma=0$ or $\sigma=n$. For $\sigma=0$, Rada \cite[Proposition 2.1]{rada2009mcclelland} showed that $D_S=D$ is acyclic. We focus on the case $\sigma=n$. Then $\lambda_i=1$ for all $i$. From $\sum_{i=1}^n \lambda^k_i=w_k$, we have $w_k=n$ for $k \geq 1$. If we exclude the closed walks that consist of just one vertex (in this case, there are $\sigma=n$ of them), then we get that the number of all other possible closed walks is zero. In particular, there is no cycle with length at least two. So $D_S$ is an acyclic digraph with full loops.
	\end{proof}
	
	\section{McClelland type bound for the energy of digraphs with self-loops}
	\label{sec4}
	
	In this section, we present the McClelland type bound for the energy of digraphs with self-loops, which generalizes \cite[Theorem 2.3]{rada2009mcclelland}: if $\sigma=0,$ then $E(D)\leq \sqrt{\frac{1}{2}n (m+c_2)}.$ 
	For a symmetric digraph $D_S=\overleftrightarrow{G_S}$, the bound in \eqref{energyinsqrt} reduces to the McClelland type bound for graphs with self-loops. Li and Liu \cite[Theorem~4]{li2024boundsenergygraphsselfloops} characterized equality in this upper bound. Cases (i)--(iii) of Theorem~\ref{DS_energy_leq_sqrt} are the corresponding symmetric cases, while case (iv) gives an additional nonsymmetric case.
		
	\begin{theorem}
		\label{DS_energy_leq_sqrt}
		Suppose $D_S$ is a digraph of order $n$, size $m$, $c_2$ 2-cycles, and $\sigma$ self-loops. Then,
		\begin{equation}
			\label{energyinsqrt}
			E(D_S)\leq \sqrt{ \frac{1}{2}n \left ( m+c_2+ 2\sigma-\frac{2\sigma^2}{n} \right) }.
		\end{equation}
		The equality holds if and only if $D_S$ is 
		\begin{enumerate}[(i)]
			\item $ n\overleftrightarrow{K_1} \text{ or } \frac{n}{2} \overleftrightarrow{K_2}.$
			\item $\frac{n}{2}\overleftrightarrow{K_1} \bigcup \frac{n}{2}\widetilde{K_1} \text{ or } \frac{n}{2}\overleftrightarrow{K_2^+}$,
			where $\overleftrightarrow{K_2^+}$ is symmetrization of the complete graph $K_2$ with one self-loop. 
			\item$\frac{n}{2}\widetilde{K_2}$ or $n\widetilde{K_1}$.
			\item $\frac{n}{4}\overrightarrow{C_3} \bigcup \frac{n}{4} \widetilde{K_1}.$
		\end{enumerate}
	\end{theorem}
	
	\begin{proof}
		By the Cauchy-Schwarz inequality 
		\begin{align*}
			\sum_{i=1}^n a_ib_i \leq \sqrt{\sum_{i=1}^n a_i^2} \sqrt{\sum_{i=1}^n b_i^2},
		\end{align*}
		with $a_i=1$ and $b_i= \left | \Re(\lambda_i)-\frac{\sigma}{n} \right|,$ we obtain
		\begin{align*}    
			E(D_S)
			&= \sum_{i=1}^n \left | \Re(\lambda_i)-\frac{\sigma}{n} \right| \\          
			&\leq \sqrt{n} \sqrt{\sum_{i=1}^n \left | \Re(\lambda_i)-\frac{\sigma}{n} \right|^2} \\      
			&= \sqrt{n} \sqrt{\sum_{i=1}^n \Re(\lambda_i)^2+\frac{\sigma^2}{n}-\frac{2\sigma}{n} \sum_{i=1}^n \Re(\lambda_i)} \\
			&\leq \sqrt{n} \sqrt{ \frac{1}{2}(m+c_2+2\sigma)-\frac{\sigma^2}{n} } \\
			&= \sqrt{ \frac{1}{2}n \left ( m+c_2+ 2\sigma-\frac{2\sigma^2}{n} \right) }.         
		\end{align*}    
		The sufficient condition can be checked by direct computation. We focus on the necessary condition. 
		
		Firstly, we show that $D_S$ is the direct sum of its strong components. Let $D_S'$ be the sub-digraph of $D_S$ obtained by deleting all the arcs that do not belong to any cycle. By Lemma \ref{charpoly_arc_deletion}, we have $\phi(D_S')=\phi(D_S)$. Let $m'$ be the size of $D_S'$. Then
		\begin{equation*}
			\sqrt{ \frac{1}{2}n \left ( m+c_2+ 2\sigma-\frac{2\sigma^2}{n} \right)}=E(D_S)=E(D_S') \leq \sqrt{ \frac{1}{2}n \left ( m'+c_2+ 2\sigma-\frac{2\sigma^2}{n} \right)}.
		\end{equation*} 
		Solving this gives us $m \leq m'$. Thus $m=m'$. Hence, $D_S$ does not contain any arc that does not belong to a cycle. 
Now consider any two distinct strong components of $D_S$, say $D_1$ and $D_2$. Suppose $(v,w)$ is an arc with $v\in V(D_1)$ and	$w\in V(D_2)$. Since $(v,w)$ belongs to a directed cycle, there is a directed path from $w$ to $v$. Together with the arc $(v,w)$, this shows that $v$ and $w$ are mutually reachable, contradicting that $D_1$ and $D_2$ are distinct strong components. Thus, such an arc cannot exist.
		Therefore, $D_S$ is the disjoint union of its strong components.
		
		Now suppose the equality holds. For all $i=1,...,n$, it holds that $\left | \Re(\lambda_i)-\frac{\sigma}{n} \right|=r$ for some $r\in \mathbb{R}$. So, $\Re(\lambda_i)$ equals to either $\frac{\sigma}{n}+r$ or $\frac{\sigma}{n}-r$. Let $l$ be the number of $\lambda_i$ such that $\Re(\lambda_i)=\frac{\sigma}{n}+r$. From $\sum_{i=1}^n \Re(\lambda_i)=\sigma$, we have 
		\[
		l\left( \frac{\sigma}{n}+r \right) + (n-l) \left( \frac{\sigma}{n}-r \right) = \sigma.
		\]
		Solving this equation gives $2rl = rn.$
		
		\begin{enumerate}
			\item Case 1: $r=0$. Then $E(D_S)=0$. By Lemma \ref{E(D_S)=0_acyclic}, $D_S$ is an acyclic digraph with no loop or full loops. In both of the cases,  we have $m=0$ from equality \eqref{energyinsqrt},. This implies that $D_S=n\overleftrightarrow{K_1}$ or $n\widetilde{K_1}$.
			
			\item Case 2: $r\neq 0$. Then, half of the eigenvalues have real part equal to $\frac{\sigma}{n}+r$ and the other half have real part equal to $\frac{\sigma}{n}-r$. Without loss of generality, we assume $\rho = \frac{\sigma}{n}+r$.
			\begin{enumerate}
				\item Subcase 2.1: All eigenvalues are real. Thus the spectrum of $D_S$ is 
				\begin{center}
					$\text{Spec($D_S$)}= \left\{  \left[\frac{\sigma}{n}+r \right] ^{\frac{n}{2}}, \left[ \frac{\sigma}{n}-r \right]^{\frac{n}{2}} \right\}.$
					
				\end{center}
				
				Subcase 2.1.1: $D_S$ contains $\overleftrightarrow{K_1}$ or $\widetilde{K_1}$. 
				If $D_S$ contains $\overleftrightarrow{K_1}$, then $\frac{\sigma}{n}-r=0$, since $\frac{\sigma}{n}+r>0$. Thus, $r=\frac{\sigma}{n}$ and $\frac{2\sigma}{n}$ is an eigenvalue of $D_S$. By Lemma \ref{rationalrootsofmonic}, $\frac{2\sigma}{n}$ must be an integer. Since $r>0$ and $\sigma<n$, we have $\frac{\sigma}{n}+r=\frac{2\sigma}{n}=1$. If $D_S$ contains $\widetilde{K_1}$, then $\frac{\sigma}{n}+r=1$, since $\frac{\sigma}{n}-r<1$.
				For $\frac{\sigma}{n}+r=1$, the spectrum of $D_S$ is 
				\begin{center}
					$\text{Spec($D_S$)}= \left\{  \left[1 \right] ^{\frac{n}{2}}, \left[ \frac{2\sigma}{n}-1 \right]^{\frac{n}{2}} \right\}.$
				\end{center}
				By Lemma \ref{rationalrootsofmonic}, $\frac{2\sigma}{n}-1$ must be an integer. This only happens when $\sigma = 0, \frac{n}{2} \text{ or } n$.
				\begin{enumerate}[i.]
					\item  If $\sigma=0$. From the spectrum of $D_S$, 
					\begin{center}
						$\text{Spec($D_S$)}= \left\{  \left[1 \right] ^{\frac{n}{2}}, \left[ -1 \right]^{\frac{n}{2}} \right\},$
					\end{center}
					we deduce that it is impossible to contain any isolated vertex $\overleftrightarrow{K_1}$.
					
					\item If $\sigma=\frac{n}{2}$, then $m=c_2=0$ from Equality \eqref{energyinsqrt}. So $D_S=\frac{n}{2} \overleftrightarrow{K_1} \bigcup \frac{n}{2} \widetilde{K_1}$.
					
					\item If $\sigma=n$, then $\frac{\sigma}{n}+r=1$ gives $r=0$,
					which contradicts $r>0$.
				\end{enumerate}
				Therefore, $D_S=\frac{n}{2}\overleftrightarrow{K_1}
				\bigcup\frac{n}{2}\widetilde{K_1}$.\\
				
				
				Subcase 2.1.2: $D_S$ does not contain $\overleftrightarrow{K_1}$ or $\widetilde{K_1}$. Suppose $D_1,\ldots, D_k$ are the strong components of $D_S$, each with order $n_i$ and $\sigma_i$ self-loops. 
				Each strong component $D_i$ has order $n_i\geq2$. If its Perron eigenvalue were $\frac{\sigma}{n}-r$, then all its eigenvalues would equal $\frac{\sigma}{n}-r$, contradicting the simplicity of its Perron eigenvalue. Thus, each strong component contributes exactly one occurrence of $\frac{\sigma}{n}+r$. Since $\frac{\sigma}{n}+r$ has multiplicity $\frac{n}{2}$, we have $k=\frac{n}{2}$. From $\sum_{i=1}^k n_i=n$ and $n_i\geq2$, it follows that $n_i=2$ for every $i$. Each $D_i$ contains both arcs between its two vertices. Moreover, its two eigenvalues sum to $\sigma_i$, so $\sigma_i=(\frac{\sigma}{n}+r)+(\frac{\sigma}{n}-r) =\frac{2\sigma}{n}$ for every $i$. Therefore, all the strong components of $D_S$ must be identical.
				Thus, $\sigma_1=\cdots=\sigma_k$ and $n_1=\cdots=n_k$. As each strong component $D_i$ has two distinct eigenvalues, by \cite[Corollary 3.5]{cavers2025distinct}, we have the following cases:
				\begin{enumerate}[(i)]
					\item $\sigma_i=0$ and $D_i=\overleftrightarrow{K_{n_i}}$ with spectrum $\{[n_i-1]^1,[-1]^{n_i-1} \}$. We have $\sigma=0$.
					Solving 
					\begin{align*}
						\frac{\sigma}{n}+r &= n_i-1, \\
						\frac{\sigma}{n}-r &= -1,
					\end{align*}
					yields $n_i=2$. From $\sum_{i=1}^k n_i=n$, it follows that $k=\frac{n}{2}$. Therefore, $D_S=\frac{n}{2} \overleftrightarrow{K_2}$.
					
					\item $\sigma_i=1$ and $D_i=\overleftrightarrow{K_2^+}$ with spectrum $\{ [\frac{1}{2}(1+\sqrt{5})]^1, [\frac{1}{2}(1-\sqrt{5})]^1 \}$. We have $\sigma=\sum_{i=1}^k \sigma_i = k$. 
					Solving 
					\begin{align*}
						\frac{k}{n}+r &=\frac{1}{2}(1+\sqrt{5}), \\
						\frac{k}{n}-r &=\frac{1}{2}(1-\sqrt{5}),
					\end{align*}
					gives $k=\frac{n}{2}$. Therefore, $D_S=\frac{n}{2}\overleftrightarrow{K_2^+}$.
					
					\item $2\leq\sigma_i\leq n_i-1$ and $\text{Spec}(D_i)=\{[\sigma_i]^1, [0]^{n_i-1}\}$. Using the same argument, it can be shown that $n_i=2$, which is a contradiction. 
					
					\item $\sigma_i=n_i$ and $D_i=\widetilde{K_{n_i}}$ with spectrum $\{[n_i]^1, [0]^{n_i-1} \}$. By the same argument, we obtain $n_i=2$ and $k=\frac{n}{2}$. Hence, $D_S=\frac{n}{2}\widetilde{K_2}$.
				\end{enumerate}
				Therefore, $D_S=\frac{n}{2} \overleftrightarrow{K_2}, \frac{n}{2}\overleftrightarrow{K_2^+} \text{ or } \frac{n}{2}\widetilde{K_2}$.\\

				\item Subcase 2.2: Suppose $\Re(\lambda_i)= \frac{\sigma}{n}-r$
				and $\Im(\lambda_i)\ne 0$ for some $i$. We show that
				$D_S=\frac{n}{4}\overrightarrow{C}_3\cup\frac{n}{4}\widetilde{K}_1$.
				Observe that each eigenvalue with real part $\frac{\sigma}{n}+r=\rho$ equals
				$\rho$, since its modulus is at most $\rho$. Thus the spectrum of $D_S$
				becomes
				\[
				\operatorname{Spec}(D_S)
				=
				\left\{
				\left[\frac{\sigma}{n}+r\right]^{\frac{n}{2}},
				\underbrace{
					\frac{\sigma}{n}-r+i\Im(\lambda_k)
				}_{\frac{n}{2} \text{ times}}
				\right\},
				\]
				where $k=1,\ldots,\frac{n}{2}$ indexes the eigenvalues with real part
				$\frac{\sigma}{n}-r$. By a simple computation, we get
				\[
				\sum_{k=1}^{n/2}\Im(\lambda_k)^2=\frac{m-c_2}{2}.
				\]
				By our assumption, we have $m>c_2$.				
				First, we show that $D_S$ could not possibly contain an isolated vertex $\overleftrightarrow{K}_1$. Such a vertex gives the eigenvalue zero. Since $\frac{\sigma}{n}+r>0$, this corresponds to $r=\frac{\sigma}{n}$.
				
				\medskip
	\noindent (i) $r=\frac{\sigma}{n}$.\\				
				Then, $\frac{2\sigma}{n}$ is a root of $\phi(\lambda)$. So, it must be an integer. Hence, $\sigma=0, \frac{n}{2}$ or $n$. The case $\sigma=0$ is impossible since $r>0$, and the case $\sigma=n$ is also impossible since	$\overleftrightarrow{K}_1$ has no self-loop. Thus, $\sigma=\frac{n}{2}$. From
				\[
				r=\sqrt{\frac{1}{2n}
					\left(m+c_2+2\sigma-\frac{2\sigma^2}{n}\right)},
				\]
				we have $m+c_2=0$, which contradicts our assumption that $m>c_2$.
				This means that $D_S$ cannot contain an isolated vertex
				$\overleftrightarrow{K}_1$. \\
											
	Now, consider a strong component of order at least two. If its Perron eigenvalue were $\frac{\sigma}{n}-r$, then all its eigenvalues would have real part $\frac{\sigma}{n}-r$ and modulus at most $\frac{\sigma}{n}-r$. Thus, all its eigenvalues would equal $\frac{\sigma}{n}-r$, contradicting the simplicity of its Perron eigenvalue. Therefore, each such strong component contains exactly one eigenvalue equal to $\frac{\sigma}{n}+r$. \\	
				
	Suppose that $D_S$ does not contain $\widetilde{K}_1$ either. Then, every strong component has order at least two, so there are at most $\frac{n}{2}$ strong components. Since $\frac{\sigma}{n}+r$ has multiplicity $\frac{n}{2}$ and each strong component contributes exactly one occurrence, thus there must be exactly $\frac{n}{2}$ strong components of order two each and has only real eigenvalues, contradicting our assumption that $D_S$ has nonreal eigenvalues. Hence, $D_S$ contains $\widetilde{K}_1$. Since $\frac{\sigma}{n}-r<1$, we have $\frac{\sigma}{n}+r=1$.
				
				\medskip
				\noindent (ii)  $r=1-\frac{\sigma}{n}$. \\				
				In this case, $\rho=\frac{\sigma}{n}+r=1$. Let $D_i$ be a strong component of order $n_i\ge 2$ with adjacency matrix $A_i$. As shown above, its Perron eigenvalue is one. Every row sum of $A_i$ is at least one. Let $\mathbf{x}$ be a positive left Perron eigenvector of $A_i$.
				Then,
				\begin{equation}\label{eq:perron1}
				\mathbf{x}^{T}A_i\mathbf{j}=\mathbf{x}^{T}\mathbf{j}.
				\end{equation}
				If any row sum of $A_i$ exceeded one, the left side of \eqref{eq:perron1} would be strictly greater than the right side, so every row sum must be one, and $D_i$ is a directed cycle. 			
				The eigenvalues of $D_i$ are the $n_i$ distinct roots of unity. Except for the eigenvalue one, all have real part $\frac{\sigma}{n}-r$. Since a vertical line meets the unit circle in at most two points, we have $n_i\le 3$. By  assumption, some eigenvalues are nonreal, so at least one strong component is $\overrightarrow{C}_3$. Thus,
				\[
				\frac{\sigma}{n}-r=-\frac{1}{2}
				\qquad\text{and}\qquad
				\frac{\sigma}{n}+r=1,
				\]
				giving $\sigma=\frac{n}{4}$. A strong component $\overleftrightarrow{K}_2$ would contribute the eigenvalue $-1$, which is impossible since $\frac{\sigma}{n} - r = -\frac12$. Therefore, every strong component is either $\overrightarrow{C}_3$ or $\widetilde{K}_1$.				
				Let $k$ be the number of strong components isomorphic to $\overrightarrow{C}_3$. Then $n=3k+\sigma$. Since $\sigma=n/4$, we have $k=n/4$. Therefore,
				\[
				D_S=\frac{n}{4}\overrightarrow{C}_3
				\cup\frac{n}{4}\widetilde{K}_1. \qedhere
				\]
			\end{enumerate}
		\end{enumerate}
	\end{proof}

	\section{Some Bounds For Spectral Radius Of Digraphs With Self-Loops And Its Complement}
	\label{sec5}
	In this section, we define the complement of a digraph with self-loops $D_S$ and give a few bounds for the spectral radius of $D_S$ and its complement. We first recall some definitions that will be frequently used throughout this section.
	
	\begin{definition} \cite{kolotilina1993perronroot}
		Let $A$ be an $n\times n$ real matrix such that $a_{ij}a_{ji}\geq0$ for all $1\leq i,j \leq n$. The geometric symmetrization of $A$ is a matrix $S(A)=(s_{ij})$ such that $s_{ij}=\sqrt{a_{ij}a_{ji}}.$ Clearly, $S(A)$ is symmetric and if $A$ is symmetric and nonnegative, then $A=S(A)$.
	\end{definition}
	
	\begin{definition}
		\cite{kolotilina1993perronroot}    
		Let $A$ and $B$ be $n\times n$ real matrices. Then, $A$ is said to be diagonally similar to $B$, denoted as $A \overset{\cD}{\sim} B$, if there exists a real diagonal matrix $\cD=Diag(d_1, \ldots, d_n)$ such that $A=\cD^{-1}B\cD$.
	\end{definition}
	
	For an $(0,1)$-matrix $A$, we have the following result:
	
	\begin{lemma}
		\label{A_diagonal similar_SA}
		Suppose $A$ is an $(0,1)$-matrix and $A$ is diagonally similar to $S(A)$. Then, $A$ is symmetric. 
	\end{lemma}
	
	\begin{proof}
		Since $A \overset{\cD}{\sim} S(A)$, we have $\cD A\cD^{-1}=S(A)$. This implies $\frac{d_i}{d_j}a_{ij}=\sqrt{a_{ij}a_{ji}}$. Suppose $a_{ij}\neq a_{ji}$. Without loss of generality, assume $a_{ij}=1$ and $a_{ji}=0.$ Then, $\frac{d_i}{d_j}=0$, which is impossible as $d_i,d_j\neq 0$. Thus, $a_{ij}=a_{ji}$ and $A$ is symmetric.
	\end{proof}

	The following inequality generalizes the Gudi\~{n}o-Rada bound in \cite{gudino2010lowerboundofspectral}.
	\begin{lemma}
		\label{lemmalowerbound_p}
		Let $A$ be the adjacency matrix of a digraph $D_S$ of order $n$, $c_2$ $2$-cycles, and $\sigma$ self-loops. Let $\rho(A)$ be the spectral radius of $A$. Then, 
		\begin{equation}
			\label{eqlowerbound_p}
			\rho(A) \geq \frac{c_2+\sigma}{n}.
		\end{equation}        
		The equality holds if and only if $D_S=\overleftrightarrow{G_S}+ \text{\{Some arcs that do not belong to cycles\}}$, where $G_S$ is a $(\frac{c_2+\sigma}{n},\frac{c_2+\sigma}{n}+1) $-bidegreed graph with self-loops such that
		\begin{align*}
			deg_{G_S}(v)= 
			\begin{cases}
				\frac{c_2+\sigma}{n}+1, &\text{if $v\in S$} ,\\
				\frac{c_2+\sigma}{n}, &\text{if $v \in \cV-S$. }
			\end{cases}
		\end{align*}
	\end{lemma}
	
	\begin{proof}
		By \cite[Theorem 1]{kolotilina1993perronroot}, we have $\rho(A)\geq \rho(S(A))$. Let $j=(1, \ldots, 1)^T$. Using the Rayleigh quotient for $S(A)$, we have 
		\begin{align*}
			\rho(A) &\geq \rho(S(A))\\
			&=\max_{v\in \R^n-\{0\}} \frac{v^TS(A)v}{v^Tv}\\
			&\geq \frac{j^TS(A)j}{j^Tj}\\
			&=\frac{c_2+\sigma}{n}.
		\end{align*}
		
Suppose the equality holds. First, suppose that $D_S$ is the disjoint union of its strong components $D_1,\ldots,D_k$, with adjacency matrices $A_1,\ldots,A_k$, respectively. Then $A$ has the following form.
		\begin{equation}
			A=\begin{bmatrix}
				
				A_1 &        &0    \\
				& \ddots                \\
				0    &        & A_k
			\end{bmatrix}.
		\end{equation}
		By \cite[Theorem 2]{kolotilina1993perronroot}, each $A_i$ is diagonally similar to $S(A_i)$ and $S(A_i)j=\rho(A)j$ for every $i$. By Lemma \ref{A_diagonal similar_SA}, each $A_i$ is symmetric. Thus, we conclude that $D_S$ is symmetric and there is a graph with self-loops $G_S$ such that $D_S=\overleftrightarrow{G_S}$.
		Since $D_S$ is symmetric, we have 
		$$deg^+(v_i)=deg^-(v_i)=\sum_{j=1}^n a_{ij}=\sum_{j=1}^n s_{ij}=\rho(A)=\frac{c_2+\sigma}{n}.$$
		Therefore, we conclude that $D_S=\overleftrightarrow{G_S}$ is $\frac{c_2+\sigma}{n}$-regular.
		
		Now if $D_S$ contains $D_1, \ldots, D_k$ strong components but not necessary disjoint union of them, by deleting all the arcs that do not belong to any cycle and applying the above argument, we conclude that $D_S=\overleftrightarrow{G_S}+ \text{\{Some arcs that do not belong to cycles\}}$.
		
Conversely, suppose that $D_S=\overleftrightarrow{G_S}
		+\{\text{some arcs that do not belong to cycles}\}$, where $G_S$
		satisfies the stated degree condition. Since each self-loop contributes	two to the degree in $G_S$ and one to the adjacency row sum, every row of $A(\overleftrightarrow{G_S})$ has sum $\frac{c_2+\sigma}{n}$. Hence,		$\rho(A(\overleftrightarrow{G_S}))=\frac{c_2+\sigma}{n}$.		
		By Lemma~\ref{charpoly_arc_deletion}, deleting the added arcs does not change the characteristic polynomial. Therefore,
		\[
		\rho(A)=\rho(A(\overleftrightarrow{G_S}))=\frac{c_2+\sigma}{n},
		\]
		and equality holds in \eqref{eqlowerbound_p}.
	\end{proof}
	
	Next, we give a lower bound of $E(D_S)$ in terms of $c_2$ and $n$. 
	\begin{lemma}
		\label{lemma_lower_bound_E(D_S)_c_2}
		Let $D_S$ be a digraph of order $n$, $c_2$ $2$-cycles and $\sigma$ self-loops. Then, 
		\begin{equation}
			\label{eq:lowerbound_E(D_S)c_2}
			E(D_S)\geq \frac{2c_2}{n}. 
		\end{equation}
		The equality holds if and only if both of the following conditions hold:
		\begin{enumerate}[(i)]
		\item $D_S=\overleftrightarrow{G_S}+ \text{\{Some arcs that do not belong to cycles\}}$, where  $G_S$ is a $(\frac{c_2+\sigma}{n},\frac{c_2+\sigma}{n}+1) $-bidegreed graph with self-loops such that
		\begin{align*}
			deg_{G_S}(v)= 
			\begin{cases}
				\frac{c_2+\sigma}{n}+1, &\text{if $v \in S$},\\
				\frac{c_2+\sigma}{n}, &\text{if $v \in \cV-S$. }
			\end{cases}
		\end{align*}
		\item $\Re(\lambda_i)\le \frac{\sigma}{n}$ for $i=2,\ldots,n$, where $\lambda_1,\ldots,\lambda_n$ are the eigenvalues of $A(D_S)$, counted with algebraic multiplicity and $\lambda_1=\rho$ is the Perron eigenvalue.
		\end{enumerate}
	\end{lemma}
	
	\begin{proof}
		The inequality follows easily from 
		\begin{equation}
			\label{longineqofE(D_S)_2c_2_n}
			E(D_S)=2\sum_{\Re(\lambda_i) > \frac{\sigma}{n}} \left( \Re(\lambda_i)-\frac{\sigma}{n} \right) \geq 2 \left(\rho-\frac{\sigma}{n} \right) \geq \frac{2c_2}{n},
		\end{equation}
		where the first equality follows from Proposition \ref{energy+} and the last inequality follows from Lemma \ref{lemmalowerbound_p}. Suppose the equality \eqref{eq:lowerbound_E(D_S)c_2} holds. From \eqref{longineqofE(D_S)_2c_2_n}, it follows that $\rho =\frac{c_2+\sigma}{n}$. By Lemma  \ref{lemmalowerbound_p}, we know $D_S=\overleftrightarrow{G_S}+ \text{\{Some arcs that do not belong to cycles\}}$, 
		where $G_S$ satisfies the stated degree condition.
		Moreover, equality in the first inequality of \eqref{longineqofE(D_S)_2c_2_n} gives $\Re(\lambda_i)\le \sigma/n$ for $i=2,\ldots,n$, since each eigenvalue other than the designated
		$\lambda_1=\rho$ with real part greater than $\sigma/n$ contributes	a strictly positive additional term to the sum. 
		
		Conversely, by Lemma~5.4, the stated structural condition gives	$\rho=(c_2+\sigma)/n$. Since $\Re(\lambda_i)\le \sigma/n$ for $i=2,\ldots,n$, the first inequality of \eqref{longineqofE(D_S)_2c_2_n} is also an equality. Thus equality	holds in \eqref{longineqofE(D_S)_2c_2_n}.
	\end{proof}
	
	\begin{remark}
	For a simple digraph $D$, Rada \cite[Cor. 2.2]{rada2010lowerenergy} obtained the lower bound 
	\[
	E(D)\geq \sqrt{2c_2}.
	\]
	When $\sigma=0$, inequality \eqref{eq:lowerbound_E(D_S)c_2} gives $E(D)\geq 2c_2/n$. For $c_2>0$, the latter bound is stronger when $c_2>n^2/2$, whereas Rada's bound is stronger
	when $c_2<n^2/2$. The two bounds are equal when $c_2=n^2/2$.
	\end{remark}

	From the above lemma, we can deduce some digraphs that satisfy the equality.
	\begin{corollary}
		Equality \eqref{eq:lowerbound_E(D_S)c_2} is satisfied if $D_S=\overleftrightarrow{K_{(\frac{n}{n-\rho}) \times (n-\rho)}}, \overleftrightarrow{K_n}, \widetilde{K}_{(\frac{n}{n-\rho}) \times (n-\rho)}$ and $\widetilde{K_n}$, where $\rho=\frac{c_2}{n}$ is the spectral radius of the underlying simple digraph $D$.
	\end{corollary}
	
	\begin{proof}
Since $D_S$ is symmetric, we know that all eigenvalues of $D_S$ are real.
		Consider first $\sigma=0$. So $\rho=\frac{c_2}{n}$. 
If $\rho=0$, then $D=n\overleftrightarrow{K_1}$ and $E(D)=E(\widetilde{D})=0$ by Lemma \ref{lemma_E(D_S)=E(D)}. Thus equality \eqref{eq:lowerbound_E(D_S)c_2} holds. Suppose now that $\rho>0$ and equality \eqref{eq:lowerbound_E(D_S)c_2} holds.
		Equality \eqref{longineqofE(D_S)_2c_2_n} implies that $\rho$ is the only positive eigenvalue and $G_S=G$ is $\frac{c_2}{n}$-regular. By \cite[Theorem 3.23]{cvetkovic1995spectra}, $G$ is connected. Suppose further that $D=\overleftrightarrow{G}$.
		
		Assume that $D$ has $0$ as its eigenvalue.
		Let us further suppose all negative eigenvalues are equal, say $\mu$. The spectrum of $D$ becomes
		\begin{align*}
			\mathrm{Spec}(D) =\{[\rho]^1, [0]^f, [\mu]^g \},
		\end{align*}
		where $f$ and $g$ are the multiplicities of $0$ and $\mu$, respectively.
		By \cite[Theorem 6.6]{cvetkovic1995spectra}, it follows that $G$ is a complete multi-partite graph $K_{\left(-\frac{\rho}{\mu}+1\right)\times (-\mu)}$. Let us determine the values of $f,g$ and $\mu$ respectively.
		
		From Equality \eqref{longineqofE(D_S)_2c_2_n}, it follows that  
		\begin{align*}
			g\mu &=- \rho.
		\end{align*}
		We also have 
		\begin{align*}
			1+f+g=n,
		\end{align*}
		as the sum of multiplicities of all eigenvalues must be $n$. Moreover, from 
		\begin{align*}
			\sum_{i=1}^n \lambda_i^2=m=n\rho,
		\end{align*}
		we obtain 
		\begin{align*}
			g\mu^2=n\rho-\rho^2.
		\end{align*}
		Solving the simultaneous equations gives us 
		\begin{align*}
			f &= \frac{n(n-\rho-1)}{n-\rho} ,\\
			g &=\frac{\rho}{n-\rho}, \\
			\mu &= -(n-\rho). 
		\end{align*}
		Therefore, $D=\overleftrightarrow{K_{(\frac{n}{n-\rho}) \times (n-\rho)}}$. 
		
		
		On the other hand, if we assume that the spectrum of $D$ does not contain $0$ as one of its eigenvalues and only have one negative eigenvalue, then the spectrum of $D$ is 
		\begin{align*}
			\mathrm{Spec}(D) =\{[\rho]^1, [\mu]^g \},
		\end{align*}
		By \cite[Theorem 6.4]{cvetkovic1995spectra}, we conclude that $D= \overleftrightarrow{K_{n}}$. For $D=\overleftrightarrow{K_{(\frac{n}{n-\rho}) \times (n-\rho)}}$, the spectrum is $\{[\rho]^1,[0]^f,[\mu]^g\}$, with $f,g,\mu$ as given above. Hence,
			\[
			E(D)=\rho-g\mu
			=\rho+\frac{\rho}{n-\rho}(n-\rho)
			=2\rho=\frac{2c_2}{n}.
			\]
		Since $K_n=K_{n\times 1}$, taking $\rho=n-1$ in the same computation gives equality for $D=\overleftrightarrow{K_n}$.	
		Since $E(\widetilde{D})=E(D)$ by Lemma \ref{lemma_E(D_S)=E(D)}, it follows that $\widetilde{K}_{(\frac{n}{n-\rho}) \times (n-\rho)}$ and $\widetilde{K}_n$ also satisfy the equality.
	\end{proof}
	
	Next, we give an upper bound for $\rho$. 
	\begin{lemma}
		\label{lemma_ub_rho}
		Let $D_S$ be a digraph of order $n\geq2$, size $m$, $c_2$ $2$-cycles and $\sigma$ self-loops. Let $\rho$ be the spectral radius of $D_S.$ Then, 
		$$\rho \leq \frac{\sigma}{n} + \sqrt{\frac{\sigma^2}{n^2}-\frac{\sigma^2}{n}+\frac{(n-1)(m+c_2+2\sigma)}{2n}}.$$
		The equality is satisfied if $D_S=\overleftrightarrow{K_n} \text{ and } \widetilde{K_n}$.
	\end{lemma}
	
	\begin{proof}
		From Lemma \ref{sum_eigenvalue}, we have 
		$$\rho-\sigma=-\sum_{i=2}^n \Re(\lambda_i).$$
		Squaring both sides and apply the Cauchy-Schwarz inequality gives
		\begin{align*}
			(\rho-\sigma)^2= \left( -\sum_{i=2}^n \Re(\lambda_i) \right)^2 \leq (n-1)\sum_{i=2}^n \Re(\lambda_i)^2.
		\end{align*}
		Therefore, 
		\begin{align*}
			\sum_{i=2}^n \Re(\lambda_i)^2 \geq \frac{(\rho-\sigma)^2}{n-1}.
		\end{align*}
		By equation \eqref{eq:ReImineq1},
		\begin{align*}
			c_2+\sigma 
			&= \rho^2 + \sum_{i=2}^n \left(\Re(\lambda_i)^2 - \Im(\lambda_i)^2\right)\\
			&\geq \rho^2 + \frac{(\rho-\sigma)^2}{n-1} - \frac{m-c_2}{2}.
		\end{align*}
		Rearranging and grouping the like terms, we obtain
		\begin{align}
			\label{quadraticrholeq0}
			2n\rho^2-4\sigma\rho+2\sigma^2-(n-1)(m+c_2+2\sigma) &\leq0.
		\end{align}
		The quadratic polynomial $f(\rho)=2n\rho^2-4\sigma\rho+2\sigma^2-(n-1)(m+c_2+2\sigma)$ has roots 
		\begin{align*}
			\rho = \frac{\sigma}{n} \pm \sqrt{\frac{\sigma^2}{n^2}-\frac{\sigma^2}{n}+\frac{(n-1)(m+c_2+2\sigma)}{2n}}.
		\end{align*}
		From inequality \eqref{quadraticrholeq0}, we have 
		\begin{align*}
			\rho \leq \frac{\sigma}{n} + \sqrt{\frac{\sigma^2}{n^2}-\frac{\sigma^2}{n}+\frac{(n-1)(m+c_2+2\sigma)}{2n}}.
		\end{align*}
		The sufficient conditions can be checked by direct computation.
	\end{proof}
	
	Now we define the complement of digraphs with self-loops and explore some of its properties.
	\begin{definition}
		\label{complement_def}
		Let $D_S$ be a digraph with self-loops with adjacency matrix $A(D_S)$ and $S\neq \emptyset$. The complement of $D_S$, denoted as $\overline{D_S}$ is defined to be $(\overline{D})_{V\setminus S}$. 
		
		In other words, if $A(\overline{D_S})$ is the adjacency of $\overline{D_S}$, then $A(D_S)+A(\overline{D_S})=J_n.$
	\end{definition}

	
	In \cite{akbari2024line}, Akbari, Jovanovi\'c and Lim established Nordhaus--Gaddum-type bounds for the spectral radius and energy of graphs with self-loops. Their complement of
	$G_S$ is $(\overline{G})_S$, which preserves the self-loop set $S$. In Definition~\ref{complement_def}, the complement of $D_S$ is $(\overline{D})_{V\setminus S}$ for	$S\ne\emptyset$. Thus, the results below concern digraphs with a different convention for the loop set of the complement.
	
	\begin{remark}	
		\begin{enumerate}		
		\item We use the term complement because the underlying simple digraph of	$\overline{D_S}$ is $\overline{D}$ for every self-loop set $S$. If $S=\emptyset$, then $\overline{D_S}$ has no self-loops and $\overline{\overline{D_S}}=D$. If $0<|S|<n$, then its self-loop set is $V \setminus S$ and 
		\[
		\overline{\overline{D_S}}
		=\overline{(\overline{D})_{V\setminus S}}
		=\overline{\overline{D}}_{V\setminus (V\setminus S)}
		=D_S. 
		\]
		If $S=V$, then $\overline{D_V}$ is loopless, so applying the complement again gives $\overline{\overline{D_V}}=D$. Hence, this convention need not be involutive.			
		\item If $D$ is a loopless digraph, then $A(D)+A(\overline{D})=J_n-I_n$. For $\sigma\in\{0,1,\ldots,n\}$, set
		\[
		\delta_\sigma=
		\begin{cases}
			0, & \text{if } \sigma\neq 0,\\
			1, & \text{if } \sigma=0.
		\end{cases}
		\]
		Then, with the convention in Definition~\ref{complement_def},
		$A(D_S)+A(\overline{D_S})=J_n-\delta_\sigma I_n$.			
		\end{enumerate}
	\end{remark}
	
	For a $r$-regular digraph with self-loops $D_S$ of order $n$, its complement $\overline{D_S}$ is $(n-r-\delta_\sigma)$-regular. Moreover, we have the following lemma.
	
	\begin{lemma}
		\label{lambda_complement_regular}
		Let $D_S$ be a $r$-regular digraph of order $n$ and $\sigma$ self-loops with adjacency matrix $A(D_S)$. If $r, \lambda_2, \ldots, \lambda_n$ are the eigenvalues of $A(D_S)$, then $n-r-\delta_\sigma,-\delta_\sigma-\lambda_2, \ldots, -\delta_\sigma-\lambda_n$ are the eigenvalues of $A(\overline{D_S})$. 
	\end{lemma}
	
	\begin{proof}
		By Definition \ref{complement_def}, we have $A(\overline{D_S})=J_n-\delta_\sigma I_n-A(D_S)$. Then, 
		\begin{align*}
			A(\overline{D_S})j &= J_nj-\delta_\sigma I_nj-A(D_S)j \\
			&= (n-r-\delta_\sigma )j.
		\end{align*}
		Thus, $n-r-\delta_\sigma $ is an eigenvalue of $\overline{D_S}$. 
		
		Let $W=\{v\in\C^n:j^Tv=0\}$. Since $j^TA(D_S)=rj^T$, $W$ is invariant under $A(D_S)$.
			Moreover, $\C^n=\langle j\rangle\oplus W$. Thus, the eigenvalues of $A(D_S)|_W$ are $\lambda_2,\ldots,\lambda_n$, counted with algebraic multiplicity.			
			If $\lambda$ is an eigenvalue of $A(D_S)|_W$ with eigenvector $v\in W$, then
			\begin{align*}
				A(\overline{D_S})v
				&= J_n v-\delta_\sigma I_nv-A(D_S)v \\
				&=(-\delta_\sigma-\lambda)v.
			\end{align*}
			Therefore, $-\delta_\sigma-\lambda$ is an eigenvalue of $A(\overline{D_S})$.			
			Since $J_n|_W=0$, we also have
			\[
			A(\overline{D_S})|_W
			=-\delta_\sigma I_W-A(D_S)|_W,
			\]
			where $I_W$ is the identity on $W$. Hence, the remaining eigenvalues of $A(\overline{D_S})$, counted with algebraic multiplicity, are $-\delta_\sigma-\lambda_2,\ldots,-\delta_\sigma-\lambda_n$.
		
%
	\end{proof}
	
	Next, we give a Nordhaus-Gaddum type bound for $\rho(D_S)$ and $\rho(\overline{D_S})$.
	
	\begin{lemma}
		\label{ul_bounds_p+p-}
		Let $D_S$ be a digraph of order $n$, size $m$ and $\sigma$ self-loops
		and $\overline{D_S}$ with size $\overline{m}$. Let $\rho$ and $\overline{\rho}$ be the spectral radius of $D_S$ and $\overline{D_S}$ respectively. Then,
		\begin{equation}
			\label{eq:ul_bounds_p+p-}
			1-\delta_\sigma+\frac{c_2+\overline{c_2}}{n} \leq \rho+ \overline{\rho} \leq 1-\delta_\sigma + \sqrt{ (n-1)^2-\frac{4\sigma(n-\sigma)}{n^2} + \frac{4\sigma(n-\sigma)}{n}+\frac{(n-1)(c_2+\overline{c_2})}{n} }.
		\end{equation}
	\end{lemma}
	
	\begin{proof}
		Suppose first $\sigma\geq1$. By definition, $m+\overline{m}=n(n-1)$. By Lemma \ref{lemmalowerbound_p}, 
		\begin{align*}
			\rho+\overline{\rho} &\geq \left( \frac{c_2}{n}+\frac{\sigma}{n} \right) + \left( \frac{\overline{c_2}}{n}+ \frac{n-\sigma}{n} \right) \\
			&=\frac{c_2+\overline{c_2}}{n} +1.
		\end{align*}
		If $n=1$, then $c_2=\overline{c_2}=0$ and $\rho+\overline{\rho}=1-\delta_\sigma$. Thus, equality holds in
			\eqref{eq:ul_bounds_p+p-}. Now suppose $n\geq2$. By Lemma \ref{lemma_ub_rho}, we have
		\begin{equation}
			\rho \leq \frac{\sigma}{n} + \sqrt{\frac{\sigma^2}{n^2}-\frac{\sigma^2}{n}+\frac{(n-1)(m+c_2+2\sigma)}{2n}}, 
		\end{equation}
		and \
		
		\begin{equation}
			\overline{\rho} \leq \frac{n-\sigma}{n} + \sqrt{\frac{(n-\sigma)^2}{n^2}-\frac{(n-\sigma)^2}{n}+\frac{(n-1)(\overline{m}+\overline{c_2}+2(n-\sigma))}{2n}}.
		\end{equation}
		Adding the above two inequalities and using the Cauchy-Schwarz inequality, we have
		\begin{align*}
			\rho+ \overline{\rho} \leq 1 + \sqrt{ (n-1)^2-\frac{4\sigma(n-\sigma)}{n^2} + \frac{4\sigma(n-\sigma)}{n}+\frac{(n-1)(c_2+\overline{c_2})}{n} }.
		\end{align*}
		Using the same argument, for $\sigma=0$, we get   
		$$\frac{c_2+\overline{c_2}}{n} \leq \rho+ \overline{\rho} \leq \sqrt{ (n-1)^2 + \frac{(n-1)(c_2+\overline{c_2})}{n} }.$$
		
		Together, we have 
		\begin{equation*}
			1-\delta_\sigma+\frac{c_2+\overline{c_2}}{n} \leq \rho+ \overline{\rho} \leq 1-\delta_\sigma + \sqrt{ (n-1)^2-\frac{4\sigma(n-\sigma)}{n^2} + \frac{4\sigma(n-\sigma)}{n}+\frac{(n-1)(c_2+\overline{c_2})}{n} }.
		\end{equation*}
	\end{proof}
	
	\begin{lemma}
		Let $D_S$ be a digraph of order $n$, size $m$, $c_2$ $2$-cycles and $\sigma$ self-loops. Then,
		\[ 
		E(D_S)+E(\overline{D_S}) \leq \sqrt{n \left(n^2-n+c_2+\overline{c_2}+4\sigma-\frac{4\sigma^2}{n} \right)  }.
		\]
		
	\end{lemma}
	
	\begin{proof}
		By Theorem \ref{DS_energy_leq_sqrt}, we have 
		$$E(D_S)\leq \sqrt{ \frac{1}{2}n \left ( m+c_2+ 2\sigma-\frac{2\sigma^2}{n} \right) },$$
		and 
		$$
		E(\overline{D_S})\leq \sqrt{ \frac{1}{2}n \left ( \overline{m}+\overline{c_2}+ 2(n-\sigma)-\frac{2(n-\sigma)^2}{n} \right) }.
		$$
		Using the Cauchy-Schwarz inequality, we have
		\begin{equation*}
			E(D_S)+E(\overline{D_S}) \leq \sqrt{n \left(n^2-n+c_2+\overline{c_2}+4\sigma-\frac{4\sigma^2}{n} \right)  }.
		\end{equation*}
	\end{proof}
	
	\begin{theorem}
		\label{sum_of_E_and_E_comp}
		Suppose $D_S$ is a $r$-regular digraph of order $n$ and $\sigma$ self-loops. Let $\mathrm{Spec}(D_S)=\{r=\lambda_1, \ldots, \lambda_n\}$. Define 
		\begin{align*}
			S_1 &= \left\{ i \in [n] : \Re(\lambda_i) \geq \frac{\sigma}{n}   \right\}, \\
			S_2^{>0} &= \left\{ i \in [n] : -\frac{\sigma}{n} < \Re(\lambda_i) < \frac{\sigma}{n} \text{ and } \Re(\lambda_i) + \frac{n-\sigma}{n} > 0  \right\}, \\
			S_2^{\leq0} &= \left\{ i \in [n] : -\frac{\sigma}{n} <\Re(\lambda_i) < \frac{\sigma}{n} \text{ and } \Re(\lambda_i) + \frac{n-\sigma}{n} \leq 0  \right\}, \\
			S_3^{>0} &= \left\{ i \in [n] : \Re(\lambda_i) \leq -\frac{\sigma}{n} \text{ and } \Re(\lambda_i) + \frac{n-\sigma}{n} > 0  \right\}, \\
			S_3^{\leq 0} &= \left\{ i \in [n] : \Re(\lambda_i) \leq -\frac{\sigma}{n} \text{ and } \Re(\lambda_i) + \frac{n-\sigma}{n} \leq 0   \right\}.
		\end{align*}
		Let $S_2= S_2^{>0} \cup S_2^{\leq 0} $ and $S_3= S_3^{>0} \cup S_3^{\leq 0} $.
		
		For $\frac{n}{2} \leq \sigma \leq n $,
		\begin{equation}
			\label{energy_sum_complement_1}
			E(D_S)+E(\overline{D_S})= 2(n-r-1)+2\sum_{S_1 \cup S_2^{\leq 0} \cup S_3 }|\Re(\lambda_i)|+\frac{2\sigma}{n}(|S_2^{\leq 0}|+|S_3|-|S_1|+1) -2(|S_2^{\leq 0}|+|S_3|).
		\end{equation}
		
		For $0\leq\sigma < \frac{n}{2}$,
		\begin{equation}
			\label{energy_sum_complement_2}
			E(D_S)+E(\overline{D_S}) = 2(n-r-1)+2\sum_{S_1 \cup S_3^{\leq 0} }|\Re(\lambda_i)|+ \frac{2\sigma}{n}(|S_3^{\leq 0}|-|S_1|+1) -2|S_3^{\leq 0}|.
		\end{equation}
	\end{theorem}
	\begin{proof}
		Suppose first $\frac{n}{2} \leq \sigma \leq n $. From Definition \ref{E(D_S)}, we have 
		\begin{align*}
			E(D_S) &= r-\frac{\sigma}{n} + \sum_{i=2}^n \left| \Re(\lambda_i)-\frac{\sigma}{n} \right| \\
			&= r-\frac{\sigma}{n} + \sum_{S_1-\{1\}} \left( \Re(\lambda_i)-\frac{\sigma}{n} \right)+ \sum_{S_2} \left(\frac{\sigma}{n} -\Re(\lambda_i) \right)+ \sum_{S_3} \left( \frac{\sigma}{n}-\Re(\lambda_i) \right),
		\end{align*}
		
		and 
		\begin{align*}
			E(\overline{D_S}) &= n-r-\frac{n-\sigma}{n} + \sum_{i=2}^n \left| \Re(-\lambda_i)-\frac{n-\sigma}{n} \right|\\
			&= n-r-\frac{n-\sigma}{n}+\sum_{i=2}^n \left| \Re(\lambda_i)+\frac{n-\sigma}{n} \right| \\
			&= n-r-\frac{n-\sigma}{n}+ \sum_{S_1-\{1\}} \left( \Re(\lambda_i)  
			+ \frac{n-\sigma}{n} \right)\\ &+ \sum_{S_2} \left| \Re(\lambda_i)+\frac{n-\sigma}{n} \right|+ \sum_{S_3} \left(-\Re(\lambda_i)-\frac{n-\sigma}{n} \right).
		\end{align*}
		Thus
		\begin{align*}
			E(D_S)+E(\overline{D_S}) 
			&= n-1+ \sum_{S_1-\{1\}}  \left( 2\Re(\lambda_i)+1-\frac{2\sigma}{n} \right) +\sum_{S_2} \left(\frac{\sigma}{n} -\Re(\lambda_i) \right) \\
			&+ \sum_{S_2} \left| \Re(\lambda_i)+\frac{n-\sigma}{n} \right| + \sum_{S_3} \left(-2\Re(\lambda_i)-1+\frac{2\sigma}{n} \right).
		\end{align*}
		
		Note that
		\begin{align*}
			&\sum_{S_2} \left(\frac{\sigma}{n} -\Re(\lambda_i) \right) + \sum_{S_2} \left| \Re(\lambda_i)+\frac{n-\sigma}{n} \right| \\
			&= \left( \sum_{S_2^{>0}} + \sum_{S_2^{\leq 0}} \right)  \left(\frac{\sigma}{n} -\Re(\lambda_i) \right) +  \sum_{S_2^{>0}} \left( \Re(\lambda_i)+\frac{n-\sigma}{n} \right)+\sum_{S_2^{\leq 0}} \left( -\Re(\lambda_i)-\frac{n-\sigma}{n} \right) \\
			&=|S_2^{>0}|+ \sum_{S_2^{\leq 0}} \left( -2\Re(\lambda_i)-1+\frac{2\sigma}{n} \right).
		\end{align*}
		
		It follows that 
		\begin{align*}
			&E(D_S)+E(\overline{D_S}) \\
			&= n-1+ \sum_{S_1-\{1\}}  \left( 2\Re(\lambda_i)+1-\frac{2\sigma}{n} \right)\\
			&+|S_2^{>0}|+ \sum_{S_2^{\leq 0}} \left( -2\Re(\lambda_i)-1+\frac{2\sigma}{n} \right) + \sum_{S_3} \left(-2\Re(\lambda_i)-1+\frac{2\sigma}{n} \right)\\
			&= n-2r-1+2\sum_{S_1}\Re(\lambda_i)+|S_1|-1- \frac{2\sigma}{n} \left( |S_1|-1 \right)+|S_2^{>0}| - 2\sum_{S_2^{\leq 0}}\Re(\lambda_i)-|S_2^{\leq 0}|+\frac{2\sigma}{n}|S_2^{\leq 0}| \\
			&-2\sum_{S_3} \Re(\lambda_i)-|S_3|+\frac{2\sigma}{n}|S_3|\\
			&=2(n-r-1)+2\sum_{S_1 \cup S_2^{\leq 0} \cup S_3 }|\Re(\lambda_i)|+\frac{2\sigma}{n}(|S_2^{\leq 0}|+|S_3|-|S_1|+1) -2(|S_2^{\leq 0}|+|S_3|).
		\end{align*}
		For $0\leq \sigma < \frac{n}{2}$, using the similar argument, \eqref{energy_sum_complement_2} is immediate.
	\end{proof}
	
	For a digraph with self-loops $D_S$, if it has  $1\leq \sigma<\frac{n}{2}$ self-loops, then its complement $\overline{D_S}$ has $n-\sigma > \frac{n}{2}$ self-loops. Hence, equation \eqref{energy_sum_complement_1} is sufficient for computation.
	
	\begin{example}
		To illustrate Theorem \ref{sum_of_E_and_E_comp}, let $\overleftrightarrow{(K_{n,n-1})_S}=(\cV_1 \cup \cV_2, \cE)$ with $ |\cV_1|=n, |\cV_2|=n-1$, be the complete bipartite digraph with $S=\cV_1$. Note that $\overleftrightarrow{(K_{n,n-1})_S}$ is $n$-regular and its complement, $\overleftrightarrow{K_n} \bigcup \widetilde{K_{n-1}}$ is $(n-1)$-regular. By \cite{akbari2023selfloop}, the spectrum of $\overleftrightarrow{(K_{n,n-1})_S}$ is 
		\begin{align*}
			\mathrm{Spec} \left( \overleftrightarrow{(K_{n,n-1})_S} \right) = \{[n]^1, [1]^{n-1}, [0]^{n-2}, [1-n]^1 \},
		\end{align*}
		while the spectrum of $\overleftrightarrow{K_n} \bigcup \widetilde{K_{n-1}}$ is 
		\begin{align*}
			\mathrm{Spec} \left(\overleftrightarrow{K_n} \bigcup \widetilde{K_{n-1}} \right) =\{ [n-1]^1, [-1]^{n-1}, [0]^{n-2}, [n-1]^1\}. 
		\end{align*}
		It follows that
		\begin{align*}
			&E\left( \overleftrightarrow{(K_{n,n-1})_S} \right) \\
			&= \left|n-\frac{n}{2n-1} \right| +  (n-1) \left|1-\frac{n}{2n-1}\right| + (n-2)\left|0-\frac{n}{2n-1}\right|+ \left|1-n-\frac{n}{2n-1}\right| \\
			&=\frac{2(3n-1)(n-1)}{2n-1},
		\end{align*}
		and 
		\begin{align*}
			&E \left(\overleftrightarrow{K_n} \bigcup \widetilde{K_{n-1}} \right) \\
			&= \left|n-1-\frac{n-1}{2n-1} \right| +  (n-1) \left|-1-\frac{n-1}{2n-1}\right| + (n-2)\left|0-\frac{n-1}{2n-1}\right|+ \left|n-1-\frac{n-1}{2n-1}\right| \\
			&= \frac{8(n-1)^2}{2n-1}.
		\end{align*}
		Therefore,
		\begin{align*}
			E\left( \overleftrightarrow{(K_{n,n-1})_S} \right)+  E \left(\overleftrightarrow{K_n} \bigcup \widetilde{K_{n-1}} \right) &=\frac{2(3n-1)(n-1)}{2n-1}+ \frac{8(n-1)^2}{2n-1}\\
			&=\frac{2(7n-5)(n-1)}{2n-1}.
		\end{align*}
		
		Since $\sigma =n > \frac{2n-1}{2}$, we will use equation \eqref{energy_sum_complement_1}. Note that 
		\begin{align*}
			S_1 &= \{1,2,\ldots,n\}, & |S_1|&=n, \\
			S_2^{>0} &= \{n+1,\ldots, 2n-2 \}, & |S_2^{>0}|&=n-2, \\
			S_2^{\leq 0} &= \emptyset, & | S_2^{\leq 0}|&=0  , \\
			S_3 &= \{2n-1\}, & |S_3|&=1. 
		\end{align*}
		The right hand side of \eqref{energy_sum_complement_1} becomes 
		\begin{align*}
			&=2((2n-1)-n-1)+2(n+(n-1)1+(1)|1-n|)+\frac{2n}{2n-1} (0+1-n+1)-2(0+1) \\
			&= \frac{2(7n-5)(n-1)}{2n-1},
		\end{align*}
		which verifies the equation.
	\end{example}


	\section*{Acknowledgment}
	Johnny Lim acknowledges the support from the Ministry of Higher Education Malaysia for Fundamental Research Grant Scheme with Project Code: 
	\linebreak 
	FRGS/1/2025/STG06/USM/02/1. The authors thank anonymous reviewers for helpful comments that leads to this improved version, including correction to Theorem~\ref{DS_energy_leq_sqrt}.\\
	
	
	\noindent \textbf{Conflicts of interest.} \quad The authors declare no conflicts of interest. \\
	\textbf{Data availability.} \quad Not applicable. \\

	\bibliography{bibliography}{}
	\bibliographystyle{amsplain}
	

\end{document}